\documentclass{amsart}
\usepackage{latexsym}
\usepackage{amstext}
\usepackage{amsmath}
\usepackage{amssymb}
\usepackage{amsthm}
\usepackage{url}
\usepackage{enumerate}

\newtheorem{theorem}{Theorem}
\newtheorem{lemma}[theorem]{Lemma}
\newtheorem{corollary}[theorem]{Corollary}

\theoremstyle{definition}

\newtheorem{remark}[theorem]{Remark}

\def \O {\mathcal{O}}

\def \v {{\bf v}}
\def \kbar {\overline{k}}
\def \Qbar {\overline{\mathbb{Q}}}

\DeclareMathOperator{\Supp}{Supp}
\DeclareMathOperator{\dv}{div}

\DeclareMathOperator{\Pic}{Pic}

\begin{document}
\bibliographystyle{amsplain}
\title[Linear forms in logarithms and integral points on varieties]{Linear forms in logarithms and integral points on higher-dimensional varieties}
\author{Aaron Levin}
\address{Department of Mathematics\\Michigan State University\\East Lansing, MI 48824}
\curraddr{}
\email{adlevin@math.msu.edu}

\begin{abstract}
We apply inequalities from the theory of linear forms in logarithms to deduce effective results on $S$-integral points on certain higher-dimensional varieties when the cardinality of $S$ is sufficiently small.  These results may be viewed as a higher-dimensional version of an effective result of Bilu on integral points on curves.  In particular, we prove a completely explicit result for integral points on certain affine subsets of the projective plane.  As an application, we generalize an effective result of Vojta on the three-variable unit equation by giving an effective solution of the polynomial unit equation $f(u,v)=w$, where $u,v,$ and $w$ are $S$-units, $|S|\leq 3$, and $f$ is a polynomial satisfying certain conditions (which are generically satisfied).  Finally, we compare our results to a higher-dimensional version of Runge's method, which has some characteristics in common with the results here.
\end{abstract}

\maketitle

\section{Introduction}

The problem of proving effective results in Diophantine questions is one of the most pervasive and basic problems in number theory.  Already in the case of curves, the fundamental finiteness theorems for integral points and rational points (Siegel's theorem and Faltings' theorem, resp.) are not known in an effective way, that is, in general there is no known algorithm to provably compute the finite sets in the conclusion of either theorem.  In certain special cases, however, effective techniques have been developed for computing integral or rational points on curves.  The most general and widely used effective methods for integral points on curves come from the theory of linear forms in logarithms, developed originally by Baker \cite{Bak}.  In higher dimensions, effective techniques have not received much attention.  A natural first step towards proving higher-dimensional effective results consists of taking the known effective techniques for curves and applying them, to the extent possible, to the higher-dimensional situation.  In \cite{LevRun}, some progress towards this goal was achieved by formulating a higher-dimensional version of an effective method of Runge for computing integral points on curves.  In this article we will consider the theory of linear forms in logarithms and applications to integral points on higher-dimensional varieties.

One of the few directions in which progress has been made on the study of integral points on higher-dimensional varieties involves varieties which, roughly speaking, have many components at infinity (e.g., \cite{Aut, Aut2, CLZ, CZ2, CZ, LevRun, Lev}).  The results given here also fit into this framework.  We prove the following effective result for integral points on higher-dimensional varieties.  

\begin{theorem}
\label{mth2}
Let $X$ be a nonsingular projective variety defined over a number field $k$.  Let $D_1, \ldots, D_n$ be effective ample divisors on $X$ defined over $k$.  Let $D=\sum_{i=1}^nD_i$.  Let $m\leq n$ be a positive integer such that for all subsets $I\subset \{1,\ldots, n\}$, $|I|=m$, the set $\cap_{i\in I} (\Supp D_i)(\kbar)$ consists of finitely many points.  Suppose that for each point $P\in (\Supp D)(\kbar)$, there exists a nonconstant rational function $\phi\in k(X)$ satisfying $P\not\in \Supp \phi$ and $\Supp \phi\subset \Supp D$.  Let $S$ be a set of places of $k$ containing the archimedean places with 
\begin{equation*}
(m-1)|S|<n.    
\end{equation*}
Let $R$ be a set of $S$-integral points on $X\setminus D$.  Suppose that $X, D_1,\ldots, D_n, D, R, S, k$ satisfy (*) in Section \ref{sres}.  Then $R$ is contained in an effectively computable proper closed subset $Z$ of $X$.
\end{theorem}

To make the meaning of ``effective" precise, we have assumed in the theorem that one can compute certain natural quantities described in Section~\ref{sres}.  An explicit description of the higher-dimensional part of $Z$ is given in Theorem \ref{mtheorem}.  If $X=C$ is a curve, then Theorem \ref{mth2} (with $m=1$) is easily seen to be equivalent to the following theorem of Bilu \cite{Bilu}.

\begin{theorem}[Bilu]
\label{BT}
Let $C\subset \mathbb{A}^n$ be an affine curve defined over a number field $k$.  Suppose that there exist two everywhere nonvanishing regular functions on $C$ with multiplicatively independent images in $k(C)^*/k^*$.  For any finite set of places $S$ of $k$ containing the archimedean places, the set $C(\O_{k,S})$ is finite and effectively computable.
\end{theorem}

Thus, Theorem \ref{mth2} may be viewed as a higher-dimensional generalization of Bilu's theorem.  We note that, as mentioned in \cite{Bilu}, when combined with finite covers and the Chevalley-Weil theorem, Theorem \ref{BT} appears to be responsible for all known ``universally effective" results on integral points on curves (results valid for all number fields $k$ and finite sets of places $S$).

As an easy consequence of Theorem \ref{mth2}, we obtain the following result for integral points on surfaces.

\begin{corollary}
\label{cos}
Let $X$ be a nonsingular projective surface defined over a number field $k$.  Let $D_1, \ldots, D_n$ be ample effective divisors on $X$, defined over $k$, that generate a subgroup of $\Pic(X)$ of rank $r$ and pairwise do not have any common components.  Let $D=\sum_{i=1}^nD_i$.  Suppose that the intersection of the supports of any $n-r$ of the divisors $D_i$ is empty.  Let $S$ be a set of places of $k$ containing the archimedean places with 
\begin{equation*}
|S|<n.    
\end{equation*}
Let $R$ be a set of $S$-integral points on $X\setminus D$.  Suppose that $X, D_1,\ldots, D_n, D, R, S, k$ satisfy (*) in Section \ref{sres}.  Then $R$ is contained in an effectively computable proper closed subset $Z$ of $X$.
\end{corollary}

The requirement, in the above results, that the number of components at infinity be large relative to the cardinality of $S$ appears prominently in Runge's method \cite{LevRun} as well.  We will compare our results with a higher-dimensional version of Runge's method in Section \ref{sRun}.

As an application of our result on surfaces, we prove an effective result on two-variable polynomials that take on $S$-unit values at $S$-unit arguments when $|S|\leq 3$.

\begin{corollary}
\label{gunit}
Let $f\in k[x,y]$ be a polynomial of degree $d>0$ such that $f(0,0)\neq 0$ and $x^d$ and $y^d$ have nonzero coefficients in $f$.  Let $S$ be a finite set of places of $k$ containing the archimedean places with $|S|\leq 3$.  Then the set of solutions to
\begin{equation*}
f(u,v)=w, \quad u,v,w\in \O_{k,S}^*,
\end{equation*}
consists of a finite effectively computable set and a finite number of infinite families of solutions where one of $u,v,$ or $w$ is constant.
\end{corollary}
The infinite families of solutions are explicitly described in Corollary \ref{gunit2} in Section~\ref{su}.

Taking $f(x,y)$ to be an appropriate affine linear polynomial, we find that Corollary \ref{gunit} generalizes an effective result of Vojta \cite{V} on the three-variable $S$-unit equation with $|S|\leq 3$.
\begin{theorem}[Vojta]
\label{TV}
Let $k$ be a number field, $S$ a finite set of places of $k$ containing the archimedean places, and $a_1,a_2,a_3\in k^*$.  If $|S|\leq 3$, then the set of solutions to the equation
\begin{align*}
a_1u_1+a_2u_2+a_3u_3&=1, \quad u_1,u_2,u_3\in \O_{k,S}^*, 
\intertext{with}
\sum_{i\neq j}a_iu_i&\neq 0, \quad j=1,2,3,
\end{align*}
is finite and effectively computable.
\end{theorem}
We note that versions of Theorem \ref{TV} with $k=\mathbb{Q}$ were also proved by Mo and Tijdeman \cite{MT} and Skinner \cite{Skin}.  Ineffectively, versions of Corollary \ref{gunit} and Theorem~\ref{TV} can be proven without any assumption on the (finite) cardinality of $S$.  For Theorem \ref{TV}, this is a special case of a well-known result on unit equations, proved independently by Evertse \cite{Ev} and van der Poorten and Schlickewei \cite{vdP2}.  In the case of Corollary~\ref{gunit}, this is an easy consequence of a result of Vojta \cite[Cor.\ 2.4.3]{Voj3} and the proof of Corollary~\ref{gunit}.  The ineffectivity here comes ultimately from usage of the Schmidt Subspace Theorem.

More generally, Vojta proved the following result for systems of unit equations.
\begin{theorem}[Vojta]
\label{TV2}
Let $m$ and $n$ be positive integers with $n>m$.  Let $(a_{ij})$ be an $m\times n$ matrix with elements in a number field $k$ such that no $m+1$ distinct columns of the matrix have rank less than $m$, and such that no column is identically zero.  Assume further that $S$ is a finite set of places of $k$, containing the archimedean places, satisfying
\begin{equation*}
(n-m-2)|S|<n.
\end{equation*}
Then the set of solutions to the system of unit equations
\begin{equation*}
a_{i1}u_1+\cdots +a_{in}u_n=0, \quad 1\leq i\leq m, \quad u_1,\ldots, u_n\in \O_{k,S}^*,
\end{equation*}
can be effectively determined.
\end{theorem}
More precisely, viewing a solution in Theorem \ref{TV2} as a point in $\mathbb{P}^{n-1}$, the set of solutions to a system of equations as in Theorem \ref{TV2} lies in finitely many proper linear subspaces of $\mathbb{P}^{n-1}$, and these solutions may be explicitly described and parametrized.  In forthcoming work, Bennett \cite{Ben} has improved the inequality on $|S|$ in Theorem \ref{TV2} to $(n-m-1)|S|<2n$.  In particular, Bennett's methods allow one to extend Theorem \ref{TV} to four-variable unit equations, that is, to effectively solve the unit equation 
\begin{equation*}
a_1u_1+a_2u_2+a_3u_3+a_4u_4=1, \quad u_1,u_2,u_3,u_4\in \O_{k,S}^*, 
\end{equation*}
where $a_1,a_2,a_3,a_4\in k^*$ and $|S|\leq 3$.  It would be interesting to determine the extent to which Bennett's methods may be applied to gain a similar improvement to the results presented here.

In Section \ref{PP}, we prove a completely explicit version of Corollary \ref{cos} when $X=\mathbb{P}^2$ is the projective plane.  
\begin{theorem}
\label{thP2}
Let $k$ be a number field of degree $\delta$ and discriminant $\Delta$.  Let $C_1, \ldots, C_n$ be distinct curves over $k$ in $\mathbb{P}^2$ such that the intersection of any $n-1$ of the curves is empty.  Let $S$ be a set of places of $k$ containing the archimedean places with $s=|S|<n$.  Then the set of integral points $\left(\mathbb{P}^2\setminus \cup_{i=1}^nC_i\right)(\O_{k,S})$ is contained in an effectively computable proper Zariski closed subset $Z$ of $\mathbb{P}^2$.  Explicitly, let $d_i=\deg C_i$, $d=\max_i d_i$, $h=\max_i h(C_i)$, and $N=\max_{v\in S}N(v)$.  Let $C_i$ be defined by $f_i\in k[x,y,z]$, $i=1,\ldots, n$.  Let $T=\cup_{i\neq j} \left(C_i\cap C_j\right)(\kbar)$, and for each point $P\in T$, let $I_P=\{i: P\not\in C_i\}$.  For $P\in T$, let 
\begin{equation*}
\Phi_P=\left\{\frac{f_{i}^{d_{j}}}{f_{j}^{d_{i}}}: i,j\in I_P\right\}.  
\end{equation*}
Then $Z$ may be taken to consist of the union of the finite set of points
\begin{equation*}
\{P\in X(k): h(P)<2^{20s+4\delta+75}d^{6s+34}\delta^{5s+8\delta-3}s^{4s+2}N^{d^2}(\log^* N)^{2s}|\Delta|^{3/2}(\log^*|\Delta|)^{3\delta}(h+1)\}
\end{equation*}
and the Zariski closure $Z'$ of the set
\begin{equation*}
\bigcup_{P\in T} \bigcap_{\phi\in \Phi_P}\{Q\in X(\kbar): \phi(Q)=\phi(P)\}.
\end{equation*}
\end{theorem}

Being more interested in the general shape of the explicit height bound in the theorem, we have made no effort here to obtain the best possible explicit bound coming from the proof of Theorem \ref{thP2} (and indeed, carefully following the proof gives a superior, but more cumbersome, expression).

Finally, we give a brief sketch of the proof of Theorem \ref{mth2}.  The proof is a generalization of the proofs of Bilu's and Vojta's results (Theorem \ref{BT} and Theorem \ref{TV}).  Let $R$ be a set of $S$-integral points on $X\setminus D$, as in Theorem \ref{mth2}, and let $P\in R$.  Let $T\subset X(\kbar)$ be the finite set of points contained in the support of $m$ or more divisors $D_i$.  Using the assumption on the cardinality of $S$, the pigeonhole principle implies that for some point $Q\in T$ and $v\in S$, $P$ is $v$-adically close to $Q$.  Our hypotheses then provide us with a nonconstant rational function $\phi\in k(X)$ with zeros and poles only in $\Supp D\setminus \{Q\}$.  Since $P\in R$, $\phi(P)$ is essentially an $S$-unit, and $\phi(P)$ is $v$-adically close to $\phi(Q)$.  Now assuming that $\phi(P)\neq \phi(Q)$ (this is where a higher-dimensional exceptional set may appear), we apply a Baker-type inequality to conclude that $\phi(P)$, and hence $P$, must have height bounded by an explicit constant.

\section{Notation and Definitions}

Let $k$ be a number field and let $S$ be a finite set of places of $k$ containing the archimedean places.  We use $\O_k$, $\O_{k,S}$, and $\O_{k,S}^*$ to denote the ring of integers of $k$, ring of $S$-integers of $k$, and group of $S$-units of $k$, respectively.  Throughout, we let $\delta=[k:\mathbb{Q}]$ be the degree of $k$, $\Delta$ the (absolute) discriminant of $k$, $R_k$ the regulator of $k$, and $R_S$ the $S$-regulator.  

Recall that we have a canonical set $M_k$ of places (or absolute values) of $k$ consisting of one place for each prime ideal $\mathfrak{p}$ of $\mathcal{O}_k$, one place for each real embedding $\sigma:k \to \mathbb{R}$, and one place for each pair of conjugate embeddings $\sigma,\overline{\sigma}:k \to \mathbb{C}$.  For $v\in M_k$, we define
\begin{equation*}
N(v)=
\begin{cases}
2 &\text{if $v$ is archimedean},\\
N(\mathfrak{p}) &\text{if $v$ corresponds to the prime $\mathfrak{p}$},
\end{cases}
\end{equation*}
where $N(\mathfrak{p})=|\O_k/\mathfrak{p}|$ is the norm of $\mathfrak{p}$.  We normalize our absolute values so that $|p|_v=\frac{1}{p}$ if $v$ corresponds to $\mathfrak{p}$ and $\mathfrak{p}$ lies above a rational prime $p$, and $|x|_v=|\sigma(x)|$ if $v$ corresponds to an embedding $\sigma$.  For $v\in M_k$, let $k_v$ denote the completion of $k$ with respect to $v$.  We set
\begin{equation*}
\|x\|_v=|x|_v^{[k_v:\mathbb{Q}_v]/[k:\mathbb{Q}]}.
\end{equation*}
A fundamental equation is the product formula
\begin{equation*}
\prod_{v\in M_k}\|x\|_v=1,
\end{equation*}
which holds for all $x\in k^*$.

For $x$ a positive real number we let
\begin{equation*}
\log^*x=\max\{\log x, 1\},
\end{equation*}
\begin{equation*}
\epsilon_v(x)=
\begin{cases}
x &\text{if $v$ is archimedean},\\
1 &\text{otherwise},
\end{cases}
\end{equation*}
and 
\begin{equation*}
\epsilon_v'(x)=\epsilon_v(x)^{[k_v:\mathbb{Q}_v]/[k:\mathbb{Q}]}.
\end{equation*}
We note that 
\begin{equation*}
\prod_{v\in M_k}\epsilon_v'(x)=x.
\end{equation*}

In this notation, for $v\in M_k$ and $x,y\in k$ we have the inequalities
\begin{align*}
|x+y|_v&\leq \epsilon_v(2)\max\{|x|_v,|y|_v\},\\
\|x+y\|_v&\leq \epsilon_v'(2)\max\{\|x\|_v,\|y\|_v\}.
\end{align*}

For $v\in M_k$ and $\alpha\in k$, we define the local height 
\begin{equation*}
h_v(\alpha)=\log \max\{\|\alpha\|_v,1\}
\end{equation*}
and the height
\begin{equation*}
h(\alpha)=\sum_{v\in M_k}h_v(\alpha).
\end{equation*}
We will frequently make the identification $\mathbb{P}^1(k)=k\cup \{\infty\}$.  More generally, for a point $P=(x_0,\ldots,x_n)\in \mathbb{P}^n(k)$, we have the absolute logarithmic height
\begin{equation*}
h(P)=\sum_{v\in M_k} \log \max\{\|x_0\|_v,\ldots,\|x_n\|_v\}.
\end{equation*}
Note that this is independent of the number field $k$ and the choice of coordinates $x_0,\ldots, x_n\in k$.

For a polynomial $f\in k[x_1,\ldots, x_n]$ and $v\in M_k$, we let $|f|_v$ denote the maximum of the absolute values of the coefficients of $f$ with respect to $v$.  We define $\|f\|_v$ similarly.  We define the height of a polynomial by
\begin{equation*}
h(f)=\sum_{v\in M_k}\|f\|_v.
\end{equation*}
This is the same as the height of the point in projective space whose coordinates are given by the coefficients of $f$.  If $\phi:\mathbb{P}^n\to \mathbb{P}^m$ is a rational map, where $\phi=(f_0,\ldots, f_m)$ and $f_0,\ldots, f_m\in k[x_0,\ldots, x_n]$ are polynomials with no common factor, then we define
\begin{equation*}
h(\phi)=\sum_{v\in M_k}\max_i \|f_i\|_v.
\end{equation*}

Let $D$ be a hypersurface in $\mathbb{P}^n$ defined by a homogeneous polynomial $f\in k[x_0,\ldots, x_n]$ of degree $d$.  We define
\begin{equation*}
h(D)=h(f).
\end{equation*}
For $v\in M_k$ and $P=(x_0,\ldots, x_n)\in\mathbb{P}^n(k)\setminus D$, $x_0,\ldots, x_n\in k$, we define the local height function
\begin{equation}
\label{hdv}
h_{D,v}(P)=\log \frac{\|f\|_v\max_i \|x_i\|_v^d}{\|f(P)\|_v}.
\end{equation}
Note that this definition is independent of the choice of the defining polynomial $f$ and the choice of the coordinates for $P$.  We let $h_D(P)=(\deg D)h(P)$.  By the product formula, if $P\in \mathbb{P}^n(k)\setminus D$, then $\sum_{v\in M_k}h_{D,v}(P)=h_D(P)$.

If $P=(x_0,\ldots, x_n), Q=(y_0,\ldots, y_n)\in \mathbb{P}^n(k)$, $x_i,y_i\in k$, $P\neq Q$, and $v\in M_k$, we define
\begin{equation*}
h_{Q,v}(P)=\frac{\max_i \|x_i\|_v \max_i \|y_i\|_v}{\max_{i,j}\|x_iy_j-x_jy_i\|_v}.
\end{equation*}

Much more generally, one can associate a height to any closed subscheme of a projective variety.  We give here a quick summary of the relevant properties of such heights and refer the reader to Silverman's paper \cite{Sil} for the general theory and details.

Let $Y$ be a closed subscheme of a projective variety $X$, both defined over $k$.  For $v\in M_k$, one can associate a local height function $h_{Y,v}:X(k)\setminus Y\to \mathbb{R}$, well-defined up to $O(1)$, and a global height function $h_Y$, well-defined up to $O(1)$, which is a sum of appropriate local height functions.  If $Y=D$ is an effective (Cartier) divisor (which we will frequently identify with the associated closed subscheme), these height functions agree with the usual height functions associated to divisors.  Local height functions satisfy the following properties: if $Y$ and $Z$ are two closed subschemes of $X$, defined over $k$, and $v\in M_k$, then up to $O(1)$,
\begin{align*}
h_{Y\cap Z,v}&=\min\{ h_{Y,v},h_{Z,v}\},\\
h_{Y+Z,v}&=h_{Y,v}+h_{Z,v},\\
h_{Y,v}&\leq h_{Z,v}, &&\text{ if }Y\subset Z,\\
h_{Y,v}&\leq c h_{Z,v}, &&\text{ if }\Supp Y\subset \Supp Z,
\end{align*}
for some constant $c>0$, where $\Supp Y$ denotes the support of $Y$.  If $\phi:W\to X$ is a morphism of projective varieties, then
\begin{equation*}
h_{Y,v}(\phi(P))=h_{\phi^*Y,v}(P), \quad \forall P\in W(k)\setminus \phi^*Y.
\end{equation*}
Here, $Y\cap Z$, $Y+Z$, $Y\subset Z$, and $\phi^*Y$ are defined in terms of the associated ideal sheaves (see \cite{Sil}).  Global height functions satisfy similar properties (except the first property above, which becomes $h_{Y\cap Z}\leq\min\{ h_{Y},h_{Z}\}+O(1)$).

Let $D$ be a divisor on a nonsingular projective variety $X$.  For a nonzero rational function $\phi\in \kbar(X)$, we let $\dv(\phi)$ denote the divisor associated to $\phi$.  We let $\Supp D$ denote the support of $D$ and $\Supp \phi=\Supp \dv(\phi)$.  Let 
\begin{equation*}
L(D)=\{\phi\in \kbar(X): \dv(\phi)+D\geq 0\}
\end{equation*}
and $h^0(D)=\dim H^0(X,\O(D))=\dim L(D)$.  If $h^0(nD)=0$ for all $n>0$, then we let $\kappa(D)=-\infty$.  Otherwise, we define the dimension of $D$ to be the integer $\kappa(D)$ such that there exist positive constants $c_1$ and $c_2$ with
\begin{equation*}
c_1 n^{\kappa(D)} \leq h^0(nD)\leq c_2 n^{\kappa(D)}
\end{equation*}
for all sufficiently divisible $n>0$.  We define a divisor $D$ on $X$ to be big if $\kappa(D)=\dim X$.

Let $D$ be an effective divisor on $X$ and $h_D=\sum_{v\in M_k}h_{D,v}$ a height function associated to $D$.  A set of points $R\subset X(k)\setminus D$ is called a set of $S$-integral points on $X\setminus D$ if there exist constants $c_v$, $v\in M_k$, such that $c_v=0$ for all but finitely many $v$, and for all $v\not\in S$,
\begin{equation*}
h_{D,v}(P)\leq c_v
\end{equation*}
for all $P\in R$.  This is well-defined, independent of how we write $X\setminus D$ \cite[Cor.~1.4.2, Th.~1.4.11]{Voj3}.  There are other essentially equivalent definitions of integrality (see, e.g., \cite[Prop. 1.4.7]{Voj3}), but since our main tools involve heights, this will be the most natural definition for our purposes.

Let $Z$ be a closed subset of $\mathbb{P}^n$ defined over $k$.  Let $S$ be a finite set of places of $k$ containing the archimedean places.  In this case there is a natural set of integral points on $\mathbb{P}^n\setminus Z$.  We define $(\mathbb{P}^n\setminus Z)(\O_{k,S})$ to be the set of points $P\in \mathbb{P}^n(k)$ such that the Zariski closures of $P$ and $Z$ in $\mathbb{P}^n_{\O_k}$ do not meet over any $v\not\in S$.  Equivalently, if $D$ is an effective divisor on $\mathbb{P}^n$, using the local height functions defined in \eqref{hdv} one easily finds that
\begin{align*}
(\mathbb{P}^n\setminus D)(\O_{k,S})&=\left\{P\in\mathbb{P}^n(k)\setminus D: h_{D,v}(P)=0, \forall v\in M_k\setminus S\right\}\\
&=\left\{P\in\mathbb{P}^n(k)\setminus D: \sum_{v\in S}h_{D,v}(P)=(\deg D)h(P)\right\}.
\end{align*}

\section{General Results}

For the purpose of clarifying our later proofs, we first collect together various elementary facts about heights.

\subsection{Heights}

Throughout, we let $X$ be a nonsingular projective variety defined over a number field $k$.  We first recall the Northcott property for heights associated to ample divisors.

\begin{lemma}
\label{L1}
Let $D$ be an ample divisor on $X$ and $c\in \mathbb{R}$.  Then the set of points $\{P\in X(k): h_D(P)<c\}$ is finite.
\end{lemma}
More generally, finiteness holds for points of $X(\kbar)$ of bounded degree and bounded ample height.  Every height is bounded by a multiple of an ample height \cite[Prop 1.2.9(f)]{Voj3}.

\begin{lemma}
\label{L2}
Let $A$ and $D$ be divisors on $X$ with $A$ ample.  Then there exists a positive integer $N$ such that
\begin{equation*}
h_D(P)<Nh_{A}(P)+O(1)
\end{equation*}
for all $P\in X(\kbar)$.
\end{lemma}

The next two lemmas give relations between the height of a point and its image under a rational map.

\begin{lemma}
\label{L3}
Let $\phi\in k(X)$ and let $P_1,\ldots, P_q\in X(k)\setminus \Supp \phi$.  Let $S$ be a finite set of places of $k$.  Then
\begin{equation*}
\sum_{i=1}^q \sum_{v\in S}  h_{P_i,v}(P)< \sum_{i=1}^q \sum_{v\in S}  h_{\phi(P_i),v}(\phi(P))+O(1)
\end{equation*}
for all $P\in X(k)\setminus \Supp \phi$ such that $\phi(P)\neq \phi(P_i)$, $i=1,\ldots, q$.
\end{lemma}

\begin{proof}
For an appropriate blow-up $\pi:\tilde{X}\to X$, where $\pi$ is an isomorphism on $\pi^{-1}(X\setminus \Supp \phi)$, $\phi$ extends to a morphism $\tilde{\phi}:\tilde{X}\to \mathbb{P}^1$ such that $\tilde{\phi}=\phi\circ \pi$ on $\pi^{-1}(X\setminus \Supp \phi)$.  For a point $P\in X(k)\setminus \Supp \phi$, we let $\tilde{P}=\pi^{-1}(P)$.  Let $P\in X(k)\setminus \Supp \phi$ be such that $\phi(P)\neq \phi(P_i)$, $i=1,\ldots, q$.  By functoriality of heights,
\begin{equation*}
\sum_{i=1}^q \sum_{v\in S}h_{\tilde{\phi}^*\phi(P_i),v}(\tilde{P})=\sum_{i=1}^q \sum_{v\in S}  h_{\phi(P_i),v}(\tilde{\phi}(\tilde{P}))+O(1).
\end{equation*}
Since $\tilde{P_i}$ is in the support of $\tilde{\phi}^*\phi(P_i)$, we have
\begin{equation*}
\sum_{i=1}^q \sum_{v\in S}h_{\tilde{P_i},v}(\tilde{P})<\sum_{i=1}^q \sum_{v\in S}h_{\tilde{\phi}^*\phi(P_i),v}(\tilde{P})+O(1).
\end{equation*}
Now the lemma follows from the above two equations, noting that $\tilde{\phi}(\tilde{P})=\phi(P)$ and by functoriality, $h_{\tilde{P_i},v}(\tilde{P})=h_{\pi^*P_i,v}(\tilde{P})=h_{P_i,v}(P)+O(1)$.
\end{proof}

\begin{lemma}
\label{L4}
Let $D$ be an effective divisor on $X$ and let $\phi\in k(X)$ be a rational function with every pole contained in $\Supp D$.  Then for some constant $c>0$,
\begin{equation*}
h(\phi(P))< ch_D(P)+O(1)
\end{equation*}
for all $P\in X(\kbar)\setminus \Supp \phi$.
\end{lemma}

\begin{proof}
We use the same notation as in the proof of Lemma \ref{L3}.  Let $P\in X(\kbar)\setminus \Supp \phi$.   By functoriality,
\begin{equation*}
h(\phi(P))=h(\tilde{\phi}(\tilde{P}))=h_{\tilde{\phi}^*\infty}(\tilde{P})+O(1).
\end{equation*}
Since $\Supp \tilde{\phi}^*\infty\subset \Supp\pi^*D$, there exists a constant $c>0$ such that 
\begin{equation*}
h_{\tilde{\phi}^*\infty}(\tilde{P})<ch_{\pi^*D}(\tilde{P})+O(1).
\end{equation*}
By functoriality again, $h_{\pi^*D}(\tilde{P})=h_D(P)+O(1)$ and the result follows.
\end{proof}

The next lemma is crucial in our later proofs.

\begin{lemma}
\label{L5}
Let $E_1,\ldots, E_m$ be effective divisors on $X$, defined over $k$, such that $\cap_{i=1}^m E_{i}$ consists of a finite number of points, all defined over $k$.  Let $v\in M_k$.  Then there exists a positive integer $N$ such that
\begin{equation*}
\min_i h_{E_{i},v}(P)\leq N\sum_{Q\in \cap_{i=1}^m E_{i}(k)}h_{Q,v}(P)+O(1),
\end{equation*}
for all $P\in X(k)\setminus \cup_i E_{i}$.
\end{lemma}

\begin{proof}
If $\cap_{j=1}^m \Supp E_{i}=\emptyset$, then in fact
\begin{equation*}
\min \{h_{D_{i_1},v}(P),\ldots, h_{D_{i_m},v}(P)\}\leq c
\end{equation*}
for some constant $c$.  This is well known and follows, for instance, from formal properties of heights since in this case $\min \{h_{D_{i_1},v},\ldots, h_{D_{i_m},v}\}$ is a local height associated to the trivial divisor.

Otherwise, let $N$ be a positive integer such that $\cap_{i=1}^m E_{i} \subset N\sum_{Q\in \cap_{i=1}^mE_{i}(k)}Q$.  Then by properties of heights,
\begin{equation*}
\min_i h_{E_{i},v}(P)=h_{\cap_{i=1}^m E_{i},v}(P)+O(1)\leq N\sum_{Q\in \cap_{i=1}^m E_{i}(k)}h_{Q,v}(P)+O(1)
\end{equation*}
for all $P\in X(k)\setminus \cup_i E_{i}$.
\end{proof}

Finally, we record two basic facts about integral points that follow from the definitions and basic properties of height functions (see also \cite[Lemma 1.4.6]{Voj3}).

\begin{lemma}
\label{L6}
Let $D_1,\ldots, D_n$ be effective divisors on $X$, defined over $k$, and let $D=\sum_{i=1}^nD_i$.  Let $S$ be a finite set of places of $k$ containing the archimedean places and let $R$ be a set of $S$-integral points on $X\setminus D$.  Then
\begin{equation*}
\sum_{v\in S}h_{D_i,v}(P)=h_{D_i}(P)+O(1), \quad i=1,\ldots, n,
\end{equation*}
for all $P\in R$.  If $\phi\in k(X)$ with $\Supp \phi\subset \Supp D$, then there exists a finite set of places $T$ of $k$ such that 
\begin{equation*}
\phi(P)\in \O_{k,T}^*
\end{equation*}
 for all $P\in R$.
\end{lemma}

\subsection{Results}
\label{sres}
Let $X$ be a nonsingular projective variety defined over a number field $k$. Let $D_1, \ldots, D_n$ be effective ample divisors on $X$, defined over $k$, and set $D=\sum_{i=1}^nD_i$.  Let $S$ be a finite set of places of $k$ containing the archimedean places and $R$ a set of $S$-integral points on $X\setminus D$.  We need a hypothesis asserting that one can effectively compute the height relations of the last section.  We say that $X, D_1,\ldots, D_n, D, R, S, k$ satisfy (*) if there are height functions associated to $D_1,\ldots, D_n, D$ and points of $X$ such that:

\begin{enumerate}
\item  The finite set in Lemma \ref{L1} is effectively computable for $D$ and any $c\in \mathbb{R}$.
\item  The positive integer $N$ and $O(1)$ in Lemma \ref{L2} are effectively computable for $D$ and $A=D_i$, $i=1,\ldots, n$.
\item  The $O(1)$ in Lemma \ref{L3} is effectively computable for $S$, any $\phi\in k(X)$ with $\Supp \phi\subset \Supp D$, and any set of points $\{P_1,\ldots, P_q\}\subset X(k)\setminus \Supp \phi$.\label{e3}
\item  The $O(1)$ in Lemma \ref{L4} is effectively computable for $D$ and any $\phi$.\label{e4}
\item  The positive integer $N$ and $O(1)$ in Lemma \ref{L5} are effectively computable for any $v\in S$ and subset $\{E_1,\ldots, E_m\}\subset \{D_1,\ldots, D_n\}$.
\item  The finite set $T$ and $O(1)$ in Lemma \ref{L6} are effectively computable for $R$, $D_1,\ldots, D_n, D$, and any $\phi$.\label{e6}
\item  The above remain true upon replacing $k$ by a finite extension of $k$ and $S$ by any finite set of places containing the set of places lying above places of $S$.\label{e7}\\
\item[] Additionally, we assume that we can compute in $\Pic(X)$ as follows:
\item  All of the relations between the images of $D_1,\ldots, D_n$ in $\Pic(X)$ are effectively computable, and for any principal divisor $E$ supported on $D_1,\ldots, D_n$ one can effectively compute a rational function $\phi\in k(X)$ with $\dv(\phi)=E$.
\end{enumerate}

Examples of varieties where (*) is satisfied (for any reasonably defined $R$) include curves, projective space, and more generally projective subvarieties of $\mathbb{P}^N$ where the divisors $D_i$ are hypersurface sections.  We explicitly work out the case $X=\mathbb{P}^2$ in Section \ref{PP}.  For curves, key algorithms include the computation of Riemann-Roch spaces \cite{Sch} and relations amongst points in the Jacobian \cite{Mas}.

The main result of this section is a slightly more explicit version of Theorem \ref{mth2}.

\begin{theorem}
\label{mtheorem}
Let $X$ be a nonsingular projective variety defined over a number field $k$.  Let $D_1, \ldots, D_n$ be effective ample divisors on $X$ defined over $k$.  Let $D=\sum_{i=1}^nD_i$.  Let $m\leq n$ be a positive integer such that for all subsets $I\subset \{1,\ldots, n\}$, $|I|=m$, the set $\cap_{i\in I} (\Supp D_i)(\kbar)$ consists of finitely many points.  Suppose that for each point $P\in (\Supp D)(\kbar)$, there exists a nonconstant rational function $\phi\in k(X)$ satisfying $P\not\in \Supp \phi$ and $\Supp \phi\subset \Supp D$.  Let $S$ be a set of places of $k$ containing the archimedean places with 
\begin{equation*}
(m-1)|S|<n.    
\end{equation*}
Let $R$ be a set of $S$-integral points on $X\setminus D$.  Suppose that $X, D_1,\ldots, D_n, D, R, S, k$ satisfy (*) in Section \ref{sres}.  Then $R$ is contained in an effectively computable proper closed subset $Z$ of $X$.  Explicitly, for a point $P\in X(\kbar)$, let
\begin{equation*}
\Phi_P=\{\phi\in k(X)^*: P\not\in \Supp \phi, \Supp \phi\subset \Supp D\}
\end{equation*}
and let
\begin{equation*}
T=\bigcup_{\substack{I\subset \{1,\ldots, n\}\\|I|=m}}\bigcap_{i\in I} (\Supp D_i)(\kbar).
\end{equation*}
Then we may take $Z$ to consist of a finite effectively computable set of points together with the Zariski closure of the set
\begin{equation*}
\bigcup_{P\in T} \bigcap_{\phi\in \Phi_P}\{Q\in X(\kbar): \phi(Q)=\phi(P)\}.
\end{equation*}
\end{theorem}

\begin{remark}
As is typical, more generally one could replace ample with big in the theorem by modifying the theorem slightly (e.g., increasing the exceptional set $Z$ to account for the base loci of certain divisors).
\end{remark}

If $D_1,\ldots, D_{r+1}$ are nontrivial effective divisors on a variety $X$ that generate a subgroup of $\Pic(X)$ of rank $r$ and pairwise do not have any common components, then there exists a nonconstant rational function $\phi$ on $X$ with all zeros and poles contained in the support of $\sum_{i=1}^{r+1}D_i$.  Using this fact to construct appropriate rational functions $\phi$ in Theorem \ref{mtheorem}, we immediately obtain the following corollary.

\begin{corollary}
\label{mcorollary}
Let $X$ be a nonsingular projective variety defined over a number field $k$.  Let $D_1, \ldots, D_n$ be ample effective divisors on $X$, defined over $k$, that generate a subgroup of $\Pic(X)$ of rank $r$ and pairwise do not have any common components.  Let $D=\sum_{i=1}^nD_i$.  Let $m\leq n$ be a positive integer such that for all subsets $I\subset \{1,\ldots, n\}$, $|I|=m$, the set $\cap_{i\in I} (\Supp D_i)(\kbar)$ consists of finitely many points. Suppose that the intersection of the supports of any $n-r$ of the divisors $D_i$ is empty.  Let $S$ be a set of places of $k$ containing the archimedean places with 
\begin{equation*}
(m-1)|S|<n.    
\end{equation*}
Let $R$ be a set of $S$-integral points on $X\setminus D$.  Suppose that $X, D_1,\ldots, D_n, D, R, S, k$ satisfy (*).  Then $R$ is contained in an effectively computable proper closed subset $Z$ of $X$.
\end{corollary}

Of particular interest is the case where $X$ is a surface.

\begin{corollary}
\label{surfcor}
Let $X$ be a nonsingular projective surface over a number field $k$.  Let $D_1, \ldots, D_n$ be ample effective divisors on $X$, defined over $k$, that generate a subgroup of $\Pic(X)$ of rank $r$ and pairwise do not have any common components.  Suppose that the intersection of the supports of any $n-r$ of the divisors $D_i$ is empty.  Let $S$ be a set of places of $k$ containing the archimedean places with 
\begin{equation*}
|S|<n.    
\end{equation*}
Let $R$ be a set of $S$-integral points on $X\setminus D$.  Suppose that $X, D_1,\ldots, D_n, D, R, S, k$ satisfy (*).  Then $R$ is contained in an effectively computable proper closed subset $Z$ of $X$.  Let
\begin{equation*}
T=\bigcup_{i\neq j} (D_i\cap D_j)(\kbar).
\end{equation*}
and let $\Phi_P$ be as in Theorem \ref{mtheorem}.  Then we may take $Z$ to consist of a finite effectively computable set of points together with the Zariski closure of the set
\begin{equation*}
\bigcup_{P\in T} \bigcap_{\phi\in \Phi_P}\{Q\in X(\kbar): \phi(Q)=\phi(P)\}.
\end{equation*}
\end{corollary}

\subsection{Proofs}
\label{SP}
The key tool in this section is the main theorem from the theory of linear forms in logarithms, which we now state in the language of heights (see Theorem \ref{Baker} for a completely explicit version).

\begin{theorem}
\label{Baker2}
Let $k$ be a number field and $S$ a finite set of places of $k$ containing the archimedean places.  Let $v\in M_k$, $\alpha\in k^*$, and $\epsilon>0$.  Then there exists an effective constant $C$ such that
\begin{equation*}
h_{\alpha,v}(x)\leq \epsilon h(x)+C
\end{equation*}
for all $x\in \O_{k,S}^*$, $x\neq \alpha$.
\end{theorem}
We note that with an ineffective constant $C$, the theorem follows easily from Roth's theorem.  Before proving Theorem \ref{mtheorem}, we prove a result which can be regarded as a higher-dimensional version of Theorem \ref{Baker2}.

\begin{theorem}
\label{B2}
Let $X$ be a nonsingular projective variety defined over a number field $k$ and let $D$ be an effective divisor on $X$ defined over $k$.  Let $\phi\in k(X)$ be a nonconstant rational function with $\Supp \phi\subset \Supp D$.  Let $S$ be a finite set of places of $k$ and $R$ a set of $S$-integral points on $X\setminus D$.  Suppose that $X, D, R, S, k$ satisfy \eqref{e3}, \eqref{e4}, \eqref{e6}, \eqref{e7} of Section \ref{sres}.  Let $P_1,\ldots, P_q\in X(k)\setminus \Supp \phi$ and $\epsilon>0$.  Then
\begin{equation*}
\sum_{i=1}^q \sum_{v\in S}  h_{P_i,v}(P)< \epsilon h_D(P)+O(1)
\end{equation*}
for all $P\in R\setminus Z$, where $Z$ is the proper closed subset of $X$ defined as the Zariski closure of the set
\begin{equation*}
\{P\in X(\kbar): \phi(P)=\phi(P_i) \text{ for some $i\in \{1,\ldots,q\}$}\}.
\end{equation*}
\end{theorem}

Here, as well as elsewhere, the implicit constant in the $O(1)$ is an effective constant.

\begin{proof}
By Lemma \ref{L6}, since $R$ is a set of $S$-integral points on $X\setminus D$, without loss of generality, after enlarging $S$ we can assume that $\phi(P)\in \O_{k,S}^*$ for all $P\in R$.  Then by Theorem \ref{Baker2},
\begin{equation*}
\sum_{i=1}^q \sum_{v\in S}  h_{\phi(P_i),v}(\phi(P))< \epsilon h(\phi(P))+O(1)
\end{equation*}
for all $P\in R\setminus Z$.  By Lemma \ref{L3},
\begin{equation*}
\sum_{i=1}^q \sum_{v\in S}  h_{P_i,v}(P)< \sum_{i=1}^q \sum_{v\in S}  h_{\phi(P_i),v}(\phi(P))+O(1)
\end{equation*}
for all $P\in X(k)\setminus (Z\cup \Supp \phi)$.  By Lemma \ref{L4}, 
\begin{equation*}
\epsilon h(\phi(P))< \epsilon c h_{D}(P)+O(1)
\end{equation*}
for some positive constant $c$ and all $P\in X(k)\setminus \Supp D$.  Replacing $\epsilon$ by $\epsilon/c$ and combining the above inequalities yields 
\begin{equation*}
\sum_{i=1}^q \sum_{v\in S}  h_{P_i,v}(P)< \epsilon h_{D}(P)+O(1)
\end{equation*}
for all $P\in R\setminus Z$.
\end{proof}

We now prove Theorem \ref{mtheorem}.

\begin{proof}[Proof of Theorem \ref{mtheorem}]
By Lemma \ref{L6}, since $R$ is a set of $S$-integral points on $X\setminus D$, we have
\begin{equation*}
\sum_{v\in S}h_{D_i,v}(P)=h_{D_i}(P)+O(1), \quad i=1,\ldots, n,
\end{equation*}
for all $P\in R$.  Let $P\in R$.  Then for each $i$, there exists a place $v\in S$ such that $h_{D_i,v}(P)\geq \frac{1}{|S|}h_{D_i}(P)+O(1)$.  Since $(m-1)|S|<n$, there exists a place $v\in S$ and distinct elements $i_1,i_2,\ldots, i_m\in\{1,\ldots, n\}$ such that
\begin{equation*}
\min \{h_{D_{i_1},v}(P),\ldots, h_{D_{i_m},v}(P)\}\geq \frac{1}{|S|}\min_j h_{D_{i_j}}(P)+O(1).
\end{equation*}
By Lemma \ref{L2}, there exists a positive integer $N$ such that
\begin{equation*}
h_D(P)\leq Nh_{D_i}(P)+O(1)
\end{equation*}
for all $i$ and all $P\in X(\kbar)$.  So for $P\in R$,
\begin{equation*}
\min \{h_{D_{i_1},v}(P),\ldots, h_{D_{i_m},v}(P)\}\geq \frac{1}{N|S|}h_D(P)+O(1).
\end{equation*}

The theorem is then a consequence of the following lemma.

\begin{lemma}
Let $m, X, D_1,\ldots, D_n, R, S, k$ be as in the hypotheses of Theorem {\rm \ref{mtheorem}}.  Let $\epsilon>0$, $v\in S$, and let $i_1,\ldots, i_m\in \{1,\ldots, n\}$ be distinct integers.  Then the set of points
\begin{equation*}
\{P\in R: \min \{h_{D_{i_1},v}(P),\ldots, h_{D_{i_m},v}(P)\}> \epsilon h_D(P)\}
\end{equation*}
is contained in an effectively computable proper closed subset $Z$ of $X$.  For $P\in X(\kbar)$, let $\Phi_P$ be the set from Theorem {\rm \ref{mtheorem}} and let
\begin{equation*}
T=\bigcap_{j=1}^m (\Supp D_{i_j})(\kbar).
\end{equation*}
Then we may take $Z$ to consist of a finite effectively computable set of points together with the Zariski closure of the set
\begin{equation*}
\bigcup_{P\in T} \bigcap_{\phi\in \Phi_P}\{Q\in X(\kbar): \phi(Q)=\phi(P)\}.
\end{equation*}
\end{lemma}

\begin{proof}
If $L$ is a finite extension of $k$ and $w$ is a place of $L$ lying above $v$, then we can define a local height function
\begin{equation*}
h_{D_i,w}(P)=\frac{[L_w:k_v]}{[L:k]}h_{D_i,v}(P)
\end{equation*}
for all $P\in X(k)\setminus D_i$.  It follows that without loss of generality, after replacing $k$ by a finite extension of $k$ and $v$ by a place lying above $v$, we may assume that every point $P\in X(\kbar)$ in the intersection $\cap_{j=1}^m \Supp D_{i_j}$ is defined over $k$ (note that by hypothesis this intersection consists of a finite number of points).

If $\cap_{j=1}^m \Supp D_{i_j}=\emptyset$, then by Lemma \ref{L5},
\begin{equation*}
\min \{h_{D_{i_1},v}(P),\ldots, h_{D_{i_m},v}(P)\}\leq C
\end{equation*}
for some effective constant $C$.  In this case, the lemma follows immediately from the fact that since $D$ is ample, the set of points $\{P\in X(k): h_D(P)<C/\epsilon\}$ is finite.

Suppose now that $\cap_{j=1}^m \Supp D_{i_j}\neq \emptyset$, in which case it consists of a finite number $q$ of points.  By Lemma \ref{L5}, there exists a positive integer $N$ such that
\begin{equation*}
\min_j h_{D_{i_j},v}(P)\leq N\sum_{Q\in \cap_{j=1}^m D_{i_j}(k)}h_{Q,v}(P)+O(1)
\end{equation*}
for all $P\in X(k)\setminus \cup_j D_{i_j}$.

Let $Q\in \cap_{j=1}^m D_{i_j}(k)$.  Note that $\Phi_Q$ is a monoid under multiplication, generated by $k^*$ and finitely many rational functions in $k(X)^*$.  Let $\epsilon>0$.  Since $R$ is a set of $S$-integral points on $X\setminus D$, applying Theorem \ref{B2} multiple times yields the inequality
\begin{equation*}
h_{Q,v}(P)< \frac{\epsilon}{2Nq}h_D(P)+O(1)
\end{equation*}
for all $P\in R\setminus Z_Q$, where $Z_Q$ is the Zariski closure of the set
\begin{equation*}
\bigcap_{\phi\in \Phi_Q}\{P\in X(\kbar): \phi(P)=\phi(Q)\}.
\end{equation*}
Summing over all points in $\cap_{j=1}^m D_{i_j}(k)$, we obtain
\begin{equation*}
\min_j h_{D_{i_j},v}(P)\leq N\sum_{Q\in \cap_{j=1}^m D_{i_j}(k)}h_{Q,v}(P)+O(1)< \frac{\epsilon}{2} h_D(P)+C
\end{equation*}
for all $P\in R\setminus Z$, where $Z=\cup_{Q\in \cap_{j=1}^m D_{i_j}(k)}Z_Q$ and $C$ is an effectively computable constant.  So if $P\in R\setminus Z$ satisfies 
\begin{equation*}
\min_j h_{D_{i_j},v}(P)> \epsilon h_D(P),
\end{equation*}
then $h_D(P)<\frac{2}{\epsilon}C$.  It follows that we have
\begin{equation*}
\left\{P\in R: \min_j h_{D_{i_j},v}(P)> \epsilon h_D(P)\right\}\subset Z\cup \left\{P\in X(k): h_D(P)<\frac{2}{\epsilon}C\right\},
\end{equation*}
where $Z$ is a proper Zariski closed subset of $X$ and the last set on the right is finite.
\end{proof}
\end{proof}

\section{An application to polynomial unit equations}
\label{su}

We prove a complete version of Corollary \ref{gunit} from the Introduction.
\begin{corollary}
\label{gunit2}
Let $f\in k[x,y]$ be a polynomial of degree $d$ such that $f(0,0)=c_0\neq 0$ and $x^d$ and $y^d$ have nonzero coefficients $c_x$ and $c_y$ in $f$, respectively.  Let $S$ be a set of places of $k$ containing the archimedean places with $|S|\leq 3$.  Then the set of solutions to
\begin{equation*}
f(u,v)=w, \quad u,v,w\in \O_{k,S}^*,
\end{equation*}
consists of a finite effectively computable set and a finite number of infinite families of solutions where one of $u,v,$ or $w$ is constant.  Let
\begin{align*}
T_1&=\{a\in \O_{k,S}^*: (x-a)|(f(x,y)-c_yy^d), c_y\in \O_{k,S}^*\},\\
T_2&=\{a\in \O_{k,S}^*: (y-a)|(f(x,y)-c_xx^d), c_x\in \O_{k,S}^*\},\\
T_3&=\{a\in \O_{k,S}^*: (y-ax)|(f(x,y)-c_0), c_0\in \O_{k,S}^*\}.
\end{align*}
Then the infinite families of solutions are
\begin{alignat*}{2}
(u,v,w)&=(a,t,c_yt^d), \quad &t\in \O_{k,S}^*, \quad\text{ for each } a\in T_1,\\
(u,v,w)&=(t,a,c_xt^d),  &t\in \O_{k,S}^*,\quad\text{ for each } a\in T_2,\\
(u,v,w)&=(t,at,c_0), &t\in \O_{k,S}^*, \quad\text{ for each }a\in T_3.
\end{alignat*}
\end{corollary}

\begin{proof}
It will be convenient to work with the homogenized polynomial 
\begin{equation*}
F(x,y,z)=z^df(x/z,y/z).  
\end{equation*}
Consider $\mathbb{P}^2$ with homogeneous coordinates $(x,y,z)$ and let $D_1,D_2,D_3,D_4$ be the curves defined by $x=0$, $y=0$, $z=0$, and $F(x,y,z)=0$, respectively.  Let $D=\sum_{i=1}^4D_i$.  Let 
\begin{equation*}
R=\{(u,v,1)\in \mathbb{P}^2(k): u,v,f(u,v)\in \O_{k,S}^*\}.
\end{equation*}
Then $R\subset (\mathbb{P}^2\setminus D)(\O_{k,S})$.  Let $\{i,j,k\}=\{1,2,3\}$ and $P\in (D_i\cap D_4)(\kbar)$.  Then, using the notation of Corollary \ref{surfcor}, the Zariski closure of $\cap_{\phi\in \Phi_P}\{Q\in X(\kbar): \phi(Q)=\phi(P)\}$ is a line through $P$ and the unique point of $D_j\cap D_k$.  Now let $P_1=(1,0,0), P_2=(0,1,0)$, and $P_3=(0,0,1)$, so that $\{P_i\}=\cap_{j\in \{1,2,3\}\setminus \{i\}}D_j(\kbar)$.  Let $Z_i$ be the Zariski closure of
\begin{equation*}
\cap_{\phi\in \Phi_{P_i}}\{Q\in X(\kbar): \phi(Q)=\phi(P_i)\}
\end{equation*}
for $i=1,2,3$.
Since $\frac{F(x,y,z)}{x^d}(P_1)=c_x$, $\frac{F(x,y,z)}{y^d}(P_2)=c_y$, and $\frac{F(x,y,z)}{z^d}(P_3)=c_0$,  it follows that we have the equations $Z_1:F(x,y,z)=c_xx^d, Z_2:F(x,y,z)=c_yy^d,$ and $Z_3:F(x,y,z)=c_0z^d$.  Let $Z$ be the closed subset of $\mathbb{P}^2$ consisting of all lines connecting points of $(D_i\cap D_j)(\kbar)$ with points of $(D_k\cap D_l)(\kbar)$, where $\{i,j,k,l\}=\{1,2,3,4\}$, together with the closed subsets $Z_1$, $Z_2$, and $Z_3$.  Then it follows from Corollary \ref{surfcor} that $R\setminus Z$ consists of a finite effectively computable set of points (in fact, an explicit height bound for points in this set follows from Theorem \ref{thP2}).  

Now let $C$ be a geometrically irreducible curve in $Z$.  If $C$ is not defined over $k$ and $C'$ is any nontrivial conjugate of $C$ over $k$, then $C(k)\subset (C\cap C')(\kbar)$, a finite effectively computable set.  In particular, $R\cap C$ is finite and effectively computable.  Assume now that $C$ is defined over $k$.  Then $R\cap C$ is a set of integral points on $C\setminus (C\cap D)$.  Consider the rational functions on $C$ given by $\phi_1=\frac{x}{z}|_C$ and $\phi_2=\frac{y}{z}|_C$.  The functions $\phi_1$ and $\phi_2$ have zeros and poles only in $C\cap D$.  If $\phi_1$ and $\phi_2$ are multiplicatively independent modulo $k^*$, then Bilu's Theorem \ref{BT} implies that $R\cap C$ is finite and effectively computable.  Suppose now that this is not the case.  Then this easily implies that $C$ is given by an equation $x^my^{n-m}=az^n$, $x^mz^{n-m}=ay^n$, or $y^mz^{n-m}=ax^n$ for some nonnegative integers $m$ and $n$ and $a\in k^*$.  Suppose first that $n\geq 2$.  Then $C$ is a component of $Z_1$, $Z_2$, or $Z_3$.  Suppose that, say, $C$ is given by $x^my^{n-m}=az^n$ and is a component of $Z_1$.  Then $F(x,y,z)=c_xx^d+(x^my^{n-m}-az^n)g(x,y,z)$ for some homogeneous polynomial $g(x,y,z)\in k[x,y,z]$.  Since $C$ is geometrically irreducible and $n\geq 2$, we must have $0<m<n$.  But from the form of $F(x,y,z)$ we then see that $y^d$ cannot have a nonzero coefficient in $F(x,y,z)$, contradicting our assumptions.  The other possible cases are similar and we conclude that $n=1$.  So $C$ is defined by a linear form $x-az$, $y-az$, or $y-ax$, for some $a\in k^*$.  

Suppose that $C$ is defined by $x-az=0$.  If $R\cap C\neq \emptyset$ then $a\in \O_{k,S}^*$, which we now assume.  Since $y^d$ must have a nonzero coefficient in $F(x,y,z)$, it follows that $C$ cannot be an irreducible component of $Z_1$ or $Z_3$.  If $C$ is an irreducible component of $Z_2$, then $f(x,y)=c_yy^d+(x-a)g(x,y)$ for some polynomial $g(x,y)\in k[x,y]$.  If $C$ connects a point of $D_i\cap D_j$ with a point of $D_k\cap D_l$, where $\{i,j,k,l\}=\{1,2,3,4\}$, then it must be that $C$ connects the unique point of $D_1\cap D_3$ with a point of $D_2\cap D_4$.  If $C$ intersects $D$ in more than two points over $\kbar$, then it follows easily again from Theorem \ref{BT} that $R\cap C$ is finite and effectively computable.  So suppose that $|(C\cap D)(\kbar)|\leq 2$, in which case $|(C\cap D)(\kbar)|=2$.  Then the fact that $C\cap D_4$ consists of a single point contained in $D_2$ implies that $f(x,y)=c_yy^d+(x-a)g(x,y)$ for some polynomial $g(x,y)\in k[x,y]$.  So in any case, $f(x,y)=c_yy^d+(x-a)g(x,y)$ for some polynomial $g(x,y)\in k[x,y]$.  Now $C\cap R\neq \emptyset$ implies that $c_y\in \O_{k,S}^*$, and in this case we find that $C\cap R=\{(a,t,1): t\in \O_{k,S}^*\}$, leading to the infinite family of solutions $(u,v,w)=(a,t,c_yt^d)$, where $t\in \O_{k,S}^*$.

The cases where $C$ is defined by $y-az$ or $y-ax$ are similar, and we are led to the classification of the infinite families in the theorem.
\end{proof}

\section{Comparison with Runge's method}
\label{sRun}

An old method of Runge \cite{Run} yields effective finiteness for the set of integral points on certain curves.  In its most basic form, Runge proved:

\begin{theorem}[Runge]
\label{Run1}
Let $f\in \mathbb{Q}[x,y]$ be an absolutely irreducible polynomial of total degree $n$.  Let $f_0$ denote the leading form of $f$, i.e., the sum of the terms of total degree $n$ in $f$.  Suppose that $f_0$ factors as $f_0=g_0h_0$, where $g_0,h_0\in \mathbb{Q}[x,y]$ are nonconstant relatively prime polynomials.  Then the set of solutions to 
\begin{equation*}
f(x,y)=0, \quad x,y\in\mathbb{Z},
\end{equation*}
is finite and effectively computable.
\end{theorem}
We will state a general higher-dimensional version of Runge's method from \cite{LevRun} (see \cite{Bom} for earlier work on curves).  Before stating a higher-dimensional version, we give some definitions which allow for varying sets of places and number fields.  It will be more convenient here to use a definition of integrality involving regular functions.  Let $V$ be a variety (not necessarily projective or affine) defined over a number field $k$.  Let $s$ be a positive integer.  We call a set  $R\subset V(\kbar)$ a set of $s$-integral points on $V$ if for every point $P\in R$ there exists a set of places $S_P$ of $k(P)$, containing the archimedean places of $k(P)$, such that $|S_P|\leq s$ and for every regular function $\phi\in \kbar(V)$ on $V$ there exists a nonzero constant $c_\phi\in k^*$, independent of $P$, such that $|c_\phi\phi(P)|_v\leq 1$ for all places $v$ of $k(P)$ not in $S_P$ (extending each place $v$ of $k(P)$ to $\kbar$ in some fixed way).  With these definitions, a higher-dimensional version of Runge's theorem is the following.

\begin{theorem}
\label{gR}
Let $X$ be a nonsingular projective variety defined over a number field $k$.  Let $D=\sum_{i=1}^rD_i$ be a divisor on $X$, with $D_1,\ldots, D_r$ effective divisors defined over $k$.  Suppose that the intersection of any $m+1$ of the supports of the divisors $D_i$ is empty.  Let $s$ be a positive integer satisfying
\begin{equation*}
ms<r.
\end{equation*}
Let $R$ be a set of $s$-integral points on $X\setminus D$.  Suppose that for every regular function $\phi\in \kbar(X)$ on $X\setminus D$, the constant $c_\phi$ in the definition of $s$-integral is effectively computable with respect to $R$.  Suppose also that one can effectively compute a basis of $L(nD_i)$ for all $n>0$ and all $i$.  Then the following statements hold.
\renewcommand{\theenumi}{(\alph{enumi})}
\renewcommand{\labelenumi}{\theenumi}
\begin{enumerate}
\item  If $\kappa(D_i)>0$ for all $i$, then $R$ is contained in an effectively computable proper Zariski closed subset $Z\subset X$.\label{R1}
\item If $D_i$ is big for all $i$, then there exists an effectively computable proper Zariski closed subset $Z\subset X$, independent of $R$, such that the set $R\setminus Z$ is finite (and effectively computable).\label{R2}
\item If $D_i$ is ample for all $i$, then $R$ is finite and effectively computable.\label{R3}
\end{enumerate}
\end{theorem}

We now briefly discuss some of the advantages and disadvantages of the higher-dimensional Runge method as compared to our results here.  To begin, in some respects the conditions on the divisors $D_i$ in Theorem \ref{gR} are weaker than the conditions required in Theorem \ref{mtheorem}.  The divisors in Theorem~\ref{gR} are not required to be ample or big (though one still needs $\kappa(D_i)>0$) and furthermore there is no linear equivalence condition present in Theorem \ref{gR}.  On the other hand, for the necessary rational functions $\phi$ to exist in Theorem \ref{mtheorem}, it is necessary that the subgroup of the Picard group generated by the divisors $D_i$ not be too large (this condition is more explicitly present in Corollary \ref{mcorollary}).  Another advantage of Theorem \ref{gR} is that the result is uniform in $|S|$, giving degeneracy of integral points even as $S$ and $k$ vary subject to an appropriate inequality.  This is also, however, a limitation of Theorem~\ref{gR}, as many results are simply not true in this generality (e.g., the unit equation $u+v=1$ likely has infinitely many solutions in rational $S$-units, $|S|\leq 3$, since for instance there are expected to be infinitely many Mersenne primes).  It is not apparent from the statement of Theorem \ref{gR}, but when Runge's method applies it also gives much smaller bounds than techniques coming from Baker's theorem.

As compared to Theorem \ref{gR}, we note that the intersection condition on the divisors $D_i$ in Theorem \ref{mtheorem} is much weaker, especially on surfaces.  For instance, the intersection condition on the divisors in Corollary \ref{surfcor} allows for highly degenerate configurations of the divisors $D_i$.  Finally, we note that even in cases where the divisors $D_i$ are in general position, the crucial inequality involving $|S|$ in Theorem~\ref{mtheorem} is superior to the inequality in Theorem~\ref{gR}.  This is particularly notable in the case of surfaces, where the superior inequality on $|S|$ is crucial, for instance, in proving Corollary \ref{gunit}.

\section{Effective inequalities}

In preparation for the next section, we recall here and prove several needed inequalities.

\subsection{Linear forms in logarithms}

The deepest effective result we need is from the theory of linear forms in logarithms.  We give a statement in terms of local heights, based on an inequality of  B\'erczes, Evertse, and Gy\H{o}ry \cite{BEG}.

\begin{theorem}
\label{Baker}
Let $k$ be a number field of degree $\delta$ and let $G$ be a finitely generated multiplicative subgroup of $k^*$ of rank $t>0$.  Let $\alpha\in k^*$ and $v\in M_{k}$.  Let $0<\epsilon<1$.  Then if $x\in G$, $x\neq \alpha$, we have
\begin{equation*}
h_{\alpha,v}(x)\leq \epsilon h(x)+c_1(\epsilon, k,G,v,\alpha)+\log 2,
\end{equation*}
where
\begin{align*}
c_1(\epsilon, k,G,v,\alpha)&=6.4\frac{c_2(\delta,t)N(v)}{\epsilon \log N(v)}Q_G\max\{h(\alpha),1\}\max \left\{\log\frac{c_2(\delta,t)N(v)}{\epsilon},\log^* Q_G\right\},\\
c_2(\delta,t)&=36(16e\delta)^{3t+5}(\log^* \delta)^2.
\end{align*}
\end{theorem}

Here, if $G$ is a finitely generated multiplicative subgroup of $\Qbar$ of rank $t>0$, then we let $Q_G$ be the minimum value of
\begin{equation*}
h(u_1)\cdots h(u_{t}),
\end{equation*}
where $u_1,\ldots, u_t$ are generators for $G$ modulo the roots of unity in $G$.  We let $Q_S=Q_{\O_{k,S}^*}$.

\begin{proof}
Let $x\in G$, $x\neq \alpha$.  Suppose first that
\begin{equation*}
h_{\alpha,v}(x)\leq \epsilon h(x)+h_v(\alpha)+h_v(1/\alpha)+h_v(2).
\end{equation*}
Then, using that $h(\alpha)=h(1/\alpha)$, we have
\begin{equation*}
h_{\alpha,v}(x)\leq \epsilon h(x)+2h(\alpha)+\log 2\leq\epsilon h(x)+c_1(\epsilon, k,G,v,\alpha)+\log 2,
\end{equation*}
as $2h(\alpha)$ is easily bounded by $c_1(\epsilon, k,G,v,\alpha)$.
Suppose now that
\begin{equation*}
h_{\alpha,v}(x)> \epsilon h(x)+h_v(\alpha)+h_v(1/\alpha)+h_v(2).
\end{equation*}

We have
\begin{align*}
h_{\alpha,v}(x)&=\log \frac{\max\{\|\alpha\|_v,1\}\max\{\|x\|_v,1\}}{\|x-\alpha\|_v}\\
&=\log \frac{\max\{\|\alpha\|_v,1\}\max\{\|\frac{x}{\alpha}\|_v,\|\frac{1}{\alpha}\|_v\}}{\|\frac{x}{\alpha}-1\|_v}.
\end{align*}

Now
\begin{equation*}
\left\|\frac{x}{\alpha}\right\|_v=\left\|\left(\frac{x}{\alpha}-1\right)+1\right\|_v\leq \epsilon_v'(2)\max\left\{\left\|\frac{x}{\alpha}-1\right\|_v,1\right\}.
\end{equation*}
It follows that
\begin{align*}
h_{\alpha,v}(x)\leq \log \frac{\epsilon_v'(2)\max\{\|\alpha\|_v,1\}\max\{\left\|\frac{x}{\alpha}-1\right\|_v,\|\frac{1}{\alpha}\|_v,1\}}{\left\|\frac{x}{\alpha}-1\right\|_v}.
\end{align*}
If 
\begin{equation*}
\max\left\{\left\|\frac{x}{\alpha}-1\right\|_v,\left\|\frac{1}{\alpha}\right\|_v,1\right\}=\left\|\frac{x}{\alpha}-1\right\|_v,
\end{equation*}
then $h_{\alpha,v}(x)\leq \log \epsilon_v'(2)\max\{\|\alpha\|_v,1\}\leq h_v(\alpha)+h_v(2)$, contradicting our assumptions.  Then we must have
\begin{align*}
h_{\alpha,v}(x)&\leq \log \frac{\epsilon_v'(2)\max\{\|\alpha\|_v,1\}\max\{\|\frac{1}{\alpha}\|_v,1\}}{\left\|\frac{x}{\alpha}-1\right\|_v}\\
&\leq h_v(2)+h_v(\alpha)+h_v(1/\alpha)-\log \left\|\frac{x}{\alpha}-1\right\|_v.
\end{align*}

So
\begin{equation*}
\log \left\|\frac{x}{\alpha}-1\right\|_v<-\epsilon h(x). 
\end{equation*}
By \cite[Th.\ 4.2]{BEG}, this implies that
\begin{equation*}
h(x)\leq c_1(\epsilon, k,G,v,\alpha).
\end{equation*}
Now we note that by Lemma \ref{lhb}, proved later in this section, for any $x\in k$, $x\neq \alpha$,
\begin{equation*}
h_{\alpha,v}(x)\leq \log 2+\sum_{v\in M_k}h_{\alpha,v}(x)=\log 2+h(x).
\end{equation*}
Thus,
\begin{equation*}
h_{\alpha,v}(x)\leq \epsilon h(x)+c_1(\epsilon,k,G,v,\alpha)+\log 2.
\end{equation*}
\end{proof}

\subsection{Hilbert's Nullstellensatz}

We will need an effective version of Hilbert's Nullstellensatz.  We use the following version, due to Masser and W\"ustholz \cite{MW}.

\begin{theorem}[Effective Hilbert's Nullstellensatz]
Let $k$ be a number field and let $p_1,\ldots,p_m, q\in \O_k[x_1,\ldots, x_n]$ be polynomials of degree at most $d\geq 1$ such that $q$ vanishes at all common zeros of $p_1,\ldots, p_m$ in $\mathbb{A}^n(\kbar)$.  Then there exists a positive integer $M\leq (8d)^{2^n}$ and polynomials $a_1,\ldots, a_m\in \O_k[x_1,\ldots x_n]$ of degrees at most $(8d)^{2^n+1}$, such that
\begin{equation*}
aq^M=a_1p_1+\cdots+a_mp_m
\end{equation*}
for some nonzero element $a\in \O_k$.  Furthermore, if 
\begin{equation*}
h_\infty=\log \max_{\substack{v\in M_k\\ v|\infty}} \{|p_1|_v,\ldots, |p_m|_v, |q|_v\}, 
\end{equation*}
then 
\begin{equation*}
\log \max_{\substack{v\in M_k\\ v|\infty}} \{|a_1|_v,\ldots, |a_m|_v, |a|_v\}\leq (8d)^{2^{n+1}-1}(h_\infty+8d\log 8d).
\end{equation*}
\end{theorem}

\begin{remark}
\label{RHN}
Applying the theorem appropriately to $\mathbb{A}^n$, it's clear that the same result holds for homogeneous polynomials $p_1,\ldots,p_m, q\in \O_k[x_1,\ldots, x_n]$ such that $q$ vanishes at all common zeros of $p_1,\ldots, p_m$ in $\mathbb{P}^{n-1}(\kbar)$.  Furthermore, in this case one can clearly choose $a_1,\ldots, a_m$ to be homogeneous polynomials with $\deg a_i=M\deg q-\deg p_i$.
\end{remark}

\subsection{Arithmetic Bezout}
\label{AB}
We will make use of the following arithmetic Bezout theorem for curves in $\mathbb{P}^2$, which is essentially a special case of a general arithmetic Bezout theorem of Philippon \cite{Phi}.

\begin{theorem}
\label{TB}
Let $C_1$ and $C_2$ be distinct curves in $\mathbb{P}^2$ over $\Qbar$.  Then
\begin{equation*}
\sum_{P\in (C_1\cap C_2)(\Qbar)} h(P)\leq (\deg C_1)h(C_2)+(\deg C_2)h(C_1)+4(\deg C_1)(\deg C_2).
\end{equation*}
\end{theorem}

\begin{proof}
We will denote the height used by Philippon in \cite{Phi} by $h_{\rm Ph}$.  By \cite[Prop. 4]{Phi},
\begin{equation*}
\sum_{P\in (C_1\cap C_2)(\Qbar)} h_{\rm Ph}(P)\leq (\deg C_1)h_{\rm Ph}(C_2)+(\deg C_2)h_{\rm Ph}(C_1).
\end{equation*}

From \cite[p. 347]{Phi}, for $i=1,2$ we have

\begin{equation*}
h_{\rm Ph}(C_i)=h_{\rm Ph}(f_i)+\frac{\deg C_i}{2}
\end{equation*}
and from the definitions of the heights, easy estimates give
\begin{equation*}
h_{\rm Ph}(f_i)\leq h(f_i)+\left(\log 2+\frac{3}{4}\right)\deg C_i=h(C_i)+\left(\log 2+\frac{3}{4}\right)\deg C_i.
\end{equation*}
So
\begin{equation*}
h_{\rm Ph}(C_i)\leq h(C_i)+2\deg C_i, \qquad i=1,2.
\end{equation*}
Finally, we note that if $P$ is a point in $\mathbb{P}^n$ then $h_{\rm Ph}(P)$ is the usual height $h(P)$ except that at the archimedean places one uses the $\ell^2$-norm.  In particular, $h(P)\leq h_{\rm Ph}(P)$.  Combining the above inequalities gives the result.
\end{proof}

\subsection{Units and regulators}



Let $k$ be a number field of degree $\delta$ and discriminant $\Delta$.  We will use the following bound on the product of the class number and the regulator, proven by Lenstra \cite[Th.~6.5]{Len}.
\begin{lemma}
Suppose that $k\neq \mathbb{Q}$.  Let $r_2$ denote the number of complex places of $k$ and let $C=\left(\frac{2}{\pi}\right)^{r_2}\sqrt{|\Delta|}$.    We have
\begin{align*}
h_kR_k\leq  \frac{C(\log C)^{\delta -1-r_2}(\delta-1+\log  C)^{r_2}}{(\delta-1)!}.
\end{align*}
\end{lemma}

Let $S$ be a finite set of places of $k$ containing the archimedean places.  Recall that $Q_S$ is the minimum value of
\begin{equation*}
h(u_1)\cdots h(u_{s-1}),
\end{equation*}
where $s=|S|$ and $u_1,\ldots, u_{s-1}$ are generators of $\O_{k,S}^*$ modulo roots of unity.  For the $S$-regulator and $Q_S$, Bugeuad and Gy\H{o}ry \cite[Lemmas 1 and 3]{BG2} gave the bounds
\begin{align*}
R_S&\leq h_kR_k\prod_{v\in S\setminus S_\infty}\log N(v),\\
Q_S&\leq \frac{((s-1)!)^2}{2^{s-2}\delta^{s-1}}R_S.
\end{align*}
More crudely, letting $s=|S|$ and $N=\max_{v\in S} N(v)$, we have the estimates
\begin{align*}
R_k\leq h_kR_k&\leq \frac{\sqrt{|\Delta|} (\frac{1}{2}\log |\Delta|)^{\delta -1-r_2}(\delta-1+\frac{1}{2}\log  |\Delta|)^{r_2}}{(\delta-1)!}\\
&\leq \frac{\sqrt{|\Delta|} (\delta-1+\frac{1}{2}\log  |\Delta|)^{\delta-1}}{(\delta-1)!}\leq \frac{\sqrt{|\Delta|}2^{\delta-1}\delta^{\delta-1}(\log^*|\Delta|)^{\delta-1}}{(\delta-1)!}\\
&\leq \delta^\delta\sqrt{|\Delta|}(\log^*|\Delta|)^{\delta-1}
\end{align*}
and
\begin{align}
R_S&\leq \delta^\delta(\log^* N)^{s-\delta/2} \sqrt{|\Delta|}(\log^*|\Delta|)^{\delta-1},\label{Rest}\\
Q_S &\leq 2^{2-s}s^{2s-4}\delta^{\delta-s+1}(\log^* N)^{s-\delta/2} \sqrt{|\Delta|}(\log^*|\Delta|)^{\delta-1}\label{Qest}.
\end{align}

\subsection{Points in projective space}
We first recall an inequality of Silverman \cite[Th.~2]{Sil2} relating the height of a point in projective space and the discriminant of its field of definition.

\begin{theorem}[Silverman]
\label{tSil}
Let $k$ be a number field of degree $\delta$ and discriminant $\Delta$.  Let $P\in \mathbb{P}^n(k)$.  Then
\begin{equation*}
\frac{\log |\Delta|}{\delta}\leq (2\delta-2)h(P)+\log \delta.
\end{equation*}
\end{theorem}

For a number field $k$ and finite set of places $S$ of $k$ containing the archimedean places, define
\begin{equation*}
c_3(k, S)=
\begin{cases}
0 &\text{ if $\delta=1$ or $s=1$},\\
\frac{2s!s^{s+\frac{1}{2}}R_S}{(\log \delta/6\delta^3)^{s-2}} &\text{ otherwise},
\end{cases}
\end{equation*}
where $s=|S|$.  If $S_\infty$ denotes the set of archimedean places of $k$, then we let $c_3(k)=c_3(k,S_\infty)$.

The next lemma describes certain choices of coordinates for a point in projective space.

\begin{lemma}
\label{arch}
Let $k$ be a number field of degree $\delta$, $s$ the number of archimedean places of $k$, and $P\in \mathbb{P}^n(k)$.  
\begin{enumerate}[(a)]
\item \label{c1} There exists a choice of homogeneous coordinates $P=(x_0,\ldots, x_n)$ such that $x_0,\ldots, x_n\in \O_k$ and for any $v\in M_k$, 
\begin{align*}
\frac{1}{s}h(P)-c_3(k) &\leq \log\max\{\|x_0\|_v,\ldots, \|x_n\|_v\}\leq \frac{1}{s}h(P)+\frac{1}{2\delta s}\log |\Delta|+c_3(k), &&\text{if $v|\infty$},\\
-\frac{1}{2\delta}\log |\Delta| &\leq \log\max\{\|x_0\|_v,\ldots, \|x_n\|_v\}\leq 0, &&\text{if $v\nmid \infty$}.
\end{align*}
\item \label{c2} There exists a choice of homogeneous coordinates $P=(x_0,\ldots, x_n)$ such that $x_0,\ldots, x_n\in \O_k$ and for any $v\in M_k$, 
\begin{align*}
0 &\leq \log\max\{\|x_0\|_v,\ldots, \|x_n\|_v\}\leq (2\delta+1)h(P)+\log \delta, &&\text{if $v|\infty$},\\
 -\delta^2 h(P)-\frac{\delta}{2}\log \delta &\leq \log\max\{\|x_0\|_v,\ldots, \|x_n\|_v\}\leq 0, && \text{if $v\nmid \infty$}.
\end{align*}

\end{enumerate}
\end{lemma}

\begin{proof}
Let $S_\infty$ denote the set of archimedean places of $k$.  The case $k=\mathbb{Q}$ follows immediately by choosing $x_0,\ldots, x_n$ to be integers with $\gcd(x_0,\ldots, x_n)=1$ and $P=(x_0,\ldots, x_n)$.  We assume from now on that $\delta>1$.
Let $P=(x_0,\ldots, x_n)$ be some choice of homogeneous coordinates with $x_0,\ldots, x_n\in \O_k$.  Let $I$ be the ideal of $\O_k$ generated by $x_0,\ldots, x_n$.  From the Minkowski bound, the ideal class of $I$ contains an (integral) ideal with norm $\leq \sqrt{|\Delta|}$.  Thus, after rescaling $x_0,\ldots, x_n$, we may assume that the norm of $I$ satisfies $N(I)\leq \sqrt{|\Delta|}$.  From the definition of the height, we have
\begin{equation*}
h(P)=\sum_{v\in M_k}\log\max\{\|x_0\|_v,\ldots, \|x_n\|_v\}=\sum_{v\in S_\infty}\log\max_i \|x_i\|_v-\frac{1}{\delta} \log N(I).
\end{equation*}
So
\begin{equation*}
h(P) \leq \sum_{v\in S_\infty}\log\max\{\|x_0\|_v,\ldots, \|x_n\|_v\}\leq h(P)+\frac{1}{2\delta}\log |\Delta|.
\end{equation*}

We first consider \eqref{c1}.  The case $s=|S_\infty|=1$ is immediate from the above, so we assume from now on that $s>1$.  Consider the image of the unit group $\O_k^*$ via the logarithmic map $\lambda:\O_k^*\mapsto \mathbb{R}^s$, $\lambda(u)=(\log \|u\|_v)_{v\in S_\infty}$.  The image is a lattice in the hyperplane of $\mathbb{R}^s$ defined by $\sum_{v\in S_\infty}x_v=0$.  From \cite[p. 5]{Haj}, there exists a fundamental domain of this lattice with diameter $\leq \frac{2s!s^{s+\frac{1}{2}}R_k}{(\log \delta/6\delta^3)^{s-2}}$.  Let $c=\sum_{v\in S_\infty}\log\max\{\|x_0\|_v,\ldots, \|x_n\|_v\}$ and consider the vector 
\begin{equation*}
\v=(\log\max\{\|x_0\|_v,\ldots, \|x_n\|_v\}-c/s)_{v\in S_\infty}.
\end{equation*}
Then there exists a unit $u\in \O_k^*$ such that

\begin{equation*}
|\v-\lambda(u)|\leq \frac{2s!s^{s+\frac{1}{2}}R_k}{(\log \delta/6\delta^3)^{s-2}}=c_3(k).
\end{equation*}
Therefore, for every $v\in S_\infty$,
\begin{equation*}
|\log\max\{\|u^{-1}x_0\|_v,\ldots, \|u^{-1}x_n\|_v\}-c/s|\leq c_3(k)
\end{equation*}
and
\begin{equation*}
\frac{1}{s}h(P)-c_3(k) \leq \log\max\{\|u^{-1}x_0\|_v,\ldots, \|u^{-1}x_n\|_v\}\leq \frac{1}{s}h(P)+\frac{1}{2\delta s}\log |\Delta|+c_3(k).
\end{equation*}
Note that if $v\nmid \infty$, we also have
\begin{equation*}
-\frac{1}{2\delta}\log |\Delta| \leq -\frac{1}{\delta}\log N(I)\leq \log\max\{\|x_0\|_v,\ldots, \|x_n\|_v\}\leq 0.
\end{equation*}

We now prove \eqref{c2}.  From our earlier choice of coordinates, we have in particular
\begin{equation*}
\sum_{v\in S_\infty}\log \|x_0\|_v=\frac{1}{\delta}\log|N^k_{\mathbb{Q}}(x_0)|\leq h(P)+\frac{1}{2\delta}\log |\Delta|.
\end{equation*}
Then after scaling by $N^k_{\mathbb{Q}}(x_0)/x_0\in \O_k$, we may take $P=(x_0,\ldots, x_n)$ where $x_0\in \mathbb{Z}$, 
\begin{equation*}
\frac{1}{\delta}\log|x_0|\leq h(P)+\frac{1}{2\delta}\log |\Delta|\leq \delta h(P)+\frac{1}{2}\log \delta
\end{equation*}
by Theorem \ref{tSil}, and $x_1,\ldots, x_n\in \O_k$.  Let $v\in S_\infty$.  Then $\log\max_i \|x_i\|_v\geq 0$ and
\begin{align*}
\log\max\{\|x_0\|_v,\ldots, \|x_n\|_v\}&=h(P)-\sum_{w\in M_k\setminus \{v\}}\log\max\{\|x_0\|_w,\ldots, \|x_n\|_w\}\\
&\leq h(P)-\sum_{w\in M_k\setminus \{v\}}\log \|x_0\|_w\leq h(P)+\log \|x_0\|_v\\
&\leq h(P)+2(\delta h(P)+\frac{1}{2}\log \delta)\leq (2\delta+1)h(P)+\log \delta.
\end{align*}

We also clearly have
\begin{equation*}
-\delta^2 h(P)-\frac{\delta}{2}\log \delta\leq -\log |x_0| \leq \log\max\{\|x_0\|_v,\ldots, \|x_n\|_v\}\leq 0
\end{equation*}
if $v$ is nonarchimedean.
\end{proof}

We also need the following result from the main theorem of \cite{Haj}, which is closely related to the previous lemma.
\begin{theorem}
\label{abl}
Let $k$ be a number field of degree $\delta$ and let $S$ be a finite set of places of $k$ containing the archimedean places.  Let $\alpha\in k$.  Then we can write
\begin{equation*}
\alpha=\beta u,
\end{equation*}
where $u\in \O_{k,S}^*$ and
\begin{align*}
h(\beta)&<s c_3(k,S)+\sum_{v\not\in S} h_v(\alpha)+\sum_{v\in S}\log \|\alpha\|_v\\
&< s c_3(k,S)+\sum_{v\not\in S} h_v(\alpha)+h_v(1/\alpha).
\end{align*}
\end{theorem}
The last inequality follows from the product formula.  This result is actually only proven in \cite{Haj} for $S$-integers $\alpha$, but the same proof given there yields the result above.

We note the estimates
\begin{align}
c_3(k,S)&\leq 2^{4s}s^{2s} \delta^{3s+\delta-6}\sqrt{|\Delta|}(\log^*|\Delta|)^{\delta-1}(\log^*N)^{s-\delta/2},\label{c3est1}\\
c_3(k)&\leq 2^{4\delta} \delta^{6\delta-6}\sqrt{|\Delta|}(\log^*|\Delta|)^{\delta-1}.\label{c3est2}
\end{align}

\subsection{Miscellaneous elementary estimates}
We have the following lower bound for heights on $\mathbb{P}^1$.

\begin{lemma}
\label{lhb}
Let $S$ be a set of places of a number field $k$.  Let $P,Q\in \mathbb{P}^1(k)$, $P\neq Q$.  Then
\begin{equation*}
\sum_{v\in S} h_{Q,v}(P)\geq -\log 2.
\end{equation*}
\end{lemma}
\begin{proof}
Let $P=(x_1,y_1), Q=(x_2,y_2)$, $x_1,x_2,y_1,y_2\in k$.  Then
\begin{align*}
h_{Q,v}(P)&=\log \frac{\max \{\|x_1\|_v,\|y_1\|_v\}\max \{\|x_2\|_v,\|y_2\|_v\}}{\|x_1y_2-x_2 y_1\|_v}\\
&\geq \log \frac{\max \{\|x_1\|_v,\|y_1\|_v\}\max \{\|x_2\|_v,\|y_2\|_v\}}{\epsilon_v'(2)\max\{\|x_1y_2\|_v,\|x_2 y_1\|_v}\\
&\geq \log \frac{\max \{\|x_1\|_v,\|y_1\|_v\}\max \{\|x_2\|_v,\|y_2\|_v\}}{\epsilon_v'(2)\max \{\|x_1\|_v,\|y_1\|_v\}\max \{\|x_2\|_v,\|y_2\|_v\}}\\
&\geq -\log \epsilon_v'(2).
\end{align*}
Therefore,
\begin{equation*}
\sum_{v\in S} h_{Q,v}(P)\geq -\log 2.
\end{equation*}
\end{proof}

We need an estimate on the height of a product of polynomials \cite[Prop.~B.7.4]{HS}.
\begin{lemma}
\label{ph}
Let $k$ be a number field.  Let $f_1,\ldots, f_m\in k[x_1,\ldots, x_n]$ be polynomials and let $f=f_1\cdots f_m$.  Then for any $v\in M_k$,
\begin{equation*}
|f|_v\leq \epsilon_v\left(\prod_{i=2}^m2^{\deg f_i}\right)\prod_{i=1}^m|f_i|_v.
\end{equation*}
In particular,
\begin{equation*}
h(f)\leq \sum_{i=1}^m h(f_i)+\left(\sum_{i=2}^m\deg f_i\right)\log 2.
\end{equation*}
\end{lemma}
For maps between projective spaces, we have the following height inequality \cite[p.~181]{HS}.
\begin{lemma}
\label{rh}
Let $\phi:\mathbb{P}^n\to \mathbb{P}^m$ be a rational map of degree $d$ defined over $\Qbar$.  Then
\begin{equation*}
h(\phi(P))\leq dh(P)+h(\phi)+\log \binom{n+d}{n}
\end{equation*}
for all $P\in \mathbb{P}^n(\Qbar)$ where $\phi$ is defined.
\end{lemma}

We also need an elementary estimate for polynomials in two variables.
\begin{lemma}
\label{inTaylor}
Let $k$ be a number field.  Let $f\in k[x,y]$ be a polynomial of degree $d$ and let $v\in M_k$.  Let $a,b,x,y\in k$ and suppose that $|x-a|_v,|y-b|_v\leq 1$.  Then
\begin{equation*}
|f(x,y)-f(a,b)|_v\leq \epsilon_v((d+2)^42^d)|f|_v\max\{|a|_v,|b|_v,1\}^d\max\{|x-a|_v,|y-b|_v\}.
\end{equation*}
\end{lemma}
\begin{proof}
Let $f(x,y)=\sum c_{ij}x^iy^j$.  Looking at the Taylor series for $f(x,y)$ around $(a,b)$ and applying the triangle inequality, we find
\begin{multline*}
|f(x,y)-f(a,b)|_v\leq\left|\sum_{m,n,m+n>0}\left(\frac{\partial^{m+n}f}{\partial x^m\partial y^n}\right)(a,b)\frac{(x-a)^m(y-b)^n}{m!n!}\right|_v\\
\leq \epsilon_v\left(\binom{d+2}{2}\right)\max_{m,n} \left|\frac{1}{m!n!}\left(\frac{\partial^{m+n}f}{\partial x^m\partial y^n}\right)(a,b)\right|_v\max\{|x-a|_v,|y-b|_v\}.
\end{multline*}
Since
\begin{align*}
\left|\frac{1}{m!n!}\left(\frac{\partial^{m+n}f}{\partial x^m\partial y^n}\right)(a,b)\right|_v&=\left|\sum c_{ij}\binom{i}{m}\binom{j}{n}a^{i-m}b^{j-n}\right|_v\\
&\leq \epsilon_v\left(\binom{d+2}{2}\right)\max_{i,j} |c_{ij}|_v\left|\binom{i}{m}\binom{j}{n}\right|_v|a^{i-m}b^{j-n}|_v\\
&\leq \epsilon_v\left(\binom{d+2}{2}2^d\right)|f|_v\max\{|a|_v,|b|_v,1\}^d,
\end{align*}
we have
\begin{equation*}
|f(x,y)-f(a,b)|_v\leq \epsilon_v((d+2)^42^d)|f|_v\max\{|a|_v,|b|_v,1\}^d\max\{|x-a|_v,|y-b|_v\}.
\end{equation*}
\end{proof}

Finally, we prove an explicit version of Lemma \ref{L3} when $X=\mathbb{P}^2$.

\begin{lemma}
\label{lemn}
Let $k$ be a number field and let $\phi\in k(\mathbb{P}^2)$ be a rational function of degree $d$ on $\mathbb{P}^2$.  Let $P,Q\in \mathbb{P}^2(k)\setminus \Supp \phi$ and $T\subset M_k$.  Suppose that $\phi(P)\neq \phi(Q)$. Then
\begin{equation*}
\sum_{v\in T}h_{Q,v}(P)\leq \sum_{v\in T} h_{\phi(Q),v}(\phi(P))+(2d+2)h(Q)+8\log(d+2)+(2d+4)\log 2.
\end{equation*}
\end{lemma}

\begin{proof}
Let $\phi=f_1/f_2$, where $f_1,f_2\in \O_k[x,y,z]$ are homogeneous polynomials of degree $d$.  Let $Q=(x_0,y_0,z_0), P=(x,y,z)$, and $\alpha=\phi(Q)$.  From the definitions,
\begin{align*}
h_{Q,v}(P)&=\log \frac{\max \{\|x_0\|_v,\|y_0\|_v,\|z_0\|_v\}\max \{\|x\|_v,\|y\|_v,\|z\|_v\}}{\max\{\|z_0x-x_0z\|_v,\|z_0y-y_0z\|_v,\|x_0y-y_0x\|_v\}},\\
h_{\alpha,v}(\phi(P))&=\log \frac{\max\{\|\alpha\|_v,1\}\max \{\|f_1(x,y,z)\|_v,\|f_2(x,y,z)\|_v\}}{\|f_1(x,y,z)-\alpha f_2(x,y,z)\|_v}.
\end{align*}
Without loss of generality, after permuting the variables, we can assume that $z_0\neq 0$ and $Q=(x_0,y_0,1)$.  If $z=0$, then
\begin{align*}
h_{Q,v}(P)&=\log \frac{\max \{\|x_0\|_v,\|y_0\|_v,1\}\max \{\|x\|_v,\|y\|_v\}}{\max\{\|x\|_v,\|y\|_v,\|x_0y-y_0x\|_v\}}\\
&\leq \log \max \{\|x_0\|_v,\|y_0\|_v,1\}.
\end{align*}
So
\begin{equation*}
\sum_{v\in T}h_{Q,v}(P)\leq \sum_{v\in T}\log \max \{\|x_0\|_v,\|y_0\|_v,1\}\leq h(Q).
\end{equation*}

Then using Lemma \ref{lhb}, in this case we have
\begin{equation*}
\sum_{v\in T}h_{Q,v}(P)\leq \sum_{v\in T}h_{\alpha,v}(P)+h(Q)+\log 2.
\end{equation*}

Suppose now that $z\neq 0$, in which case we can take $P=(x,y,1)$, for some $x,y\in k$.  

First suppose that
\begin{equation*}
\max\{|x-x_0|_v,|y-y_0|_v\}< \frac{1}{\epsilon_v((d+2)^42^{d+1})}\min_{j=1,2} \frac{|f_j(x_0,y_0,1)|_v}{|f_j|_v\max\{|x_0|_v,|y_0|_v,1\}^{d}}.
\end{equation*}
In particular, $\max\{|x-x_0|_v,|y-y_0|_v\}\leq 1$.  Let $F(u,v)=f_1(u,v,1)-\alpha f_2(u,v,1)$.  Note that $\deg F\leq d$.  From the definition of $\alpha$, $F(x_0,y_0)=0$.  Then by Lemma \ref{inTaylor}, with $a=x_0, b=y_0$, we have
\begin{equation}
\label{Feq}
|F(x,y)|_v\leq \epsilon_v((d+2)^42^{d})|F|_v\max\{|x_0|_v,|y_0|_v,1\}^{d}\max\{|x-x_0|_v,|y-y_0|_v\}.
\end{equation}
For $j=1,2$, using Lemma \ref{inTaylor} again, we find, if $v$ is archimedean,
\begin{align*}
|f_j(x,y,1)|_v&\geq |f_j(x_0,y_0,1)|_v-(d+2)^42^d|f_j|_v\max\{|x_0|_v,|y_0|_v,1\}^{d}\max\{|x-x_0|_v,|y-y_0|_v\}\\
&\geq \frac{1}{2}|f_j(x_0,y_0,1)|_v.
\end{align*}
By the same reasoning, if $v$ is nonarchimedean we have
\begin{equation*}
|f_j(x,y,1)-f_j(x_0,y_0,1)|_v<|f_j(x_0,y_0,1)|_v,
\end{equation*}
and so
\begin{equation*}
|f_j(x,y,1)|_v=|f_j(x_0,y_0,1)|_v, \quad j=1,2.
\end{equation*}
Then in any case,
\begin{equation*}
|f_j(x,y,1)|_v\geq \frac{1}{\epsilon_v(2)}|f_j(x_0,y_0,1)|_v, \quad j=1,2.
\end{equation*}

Since $\max\{|x-x_0|_v,|y-y_0|_v\}\leq 1$, we also have
\begin{equation*}
\max \{|x|_v,|y|_v,1\}\leq \epsilon_v(2)\max \{|x_0|_v,|y_0|_v,1\}.
\end{equation*}
Then
\begin{align*}
h_{Q,v}(P)&=\log \frac{\max \{\|x_0\|_v,\|y_0\|_v,1\}\max \{\|x\|_v,\|y\|_v,1\}}{\max\{\|x-x_0\|_v,\|y-y_0\|_v,\|x_0y-y_0x\|_v\}}\\
&\leq 2\log \max \{\|x_0\|_v,\|y_0\|_v,1\}+\log \epsilon_v'(2)-\log \max\{\|x-x_0\|_v,\|y-y_0\|_v\}
\end{align*}
and

\begin{align*}
h_{\alpha,v}(\phi(P))&=\log \frac{\max\{\|\alpha\|_v,1\}\max \{\|f_1(x,y,1)\|_v,\|f_2(x,y,1)\|_v\}}{\|f_1(x,y,1)-\alpha f_2(x,y,1)\|_v}\\
&= \log \max_{j=1,2} \|f_j(x,y,1)\|_v+\log \max\{\|\alpha\|_v,1\}-\log \|F(x,y)\|_v\\
&\geq \log \max_{j=1,2} \|f_j(x_0,y_0,1)\|_v+\log \max\{\|\alpha\|_v,1\}-\epsilon_v'(\log (d+2)^42^{d+1})\\
&-\log \|F\|_v- d\log \max\{\|x_0\|_v,\|y_0\|_v,1\}-\log\max\{\|x-x_0\|_v,\|y-y_0\|_v\}.
\end{align*}
by \eqref{Feq}.  We can estimate
\begin{align*}
|F|_v&=|f_1-\alpha f_2|_v\leq \epsilon_v(2)\max\{|f_1|_v,|f_2|_v\}\max \{|\alpha |_v,1\}.
\end{align*}
So
\begin{align*}
h_{\alpha ,v}(\phi(P))&\geq \log \max_{j=1,2} \|f_j(x_0,y_0,1)\|_v-\epsilon_v'(\log (d+2)^42^{d+2})-\log \max_{j=1,2}\|f_j\|_v\\
&- d\log \max\{\|x_0\|_v,\|y_0\|_v,1\}-\log\max\{\|x-x_0\|_v,\|y-y_0\|_v\}.
\end{align*}
Note that
\begin{equation*}
|f_j(x_0,y_0,1)|_v\leq \epsilon_v\left(\binom{d+2}{2}\right)|f_j|_v\max\{|x_0|_v,|y_0|_v,1\}^{d}, \quad j=1,2.
\end{equation*}
This implies that
\begin{align*}
&\sum_{v\in T}\log \frac {\max \{\|f_1(x_0,y_0,1)\|_v,\|f_2(x_0,y_0,1)\|_v\}}{\max\{\|f_1\|_v,\|f_2\|_v\}\max\{\|x_0\|_v,\|y_0\|_v,1\}^{d}}
 \\
&\geq \sum_{j=1}^2\sum_{v\in M_k}\log \frac {\|f_j(x_0,y_0,1)\|_v}{\|f_j\|_v\max\{\|x_0\|_v,\|y_0\|_v,1\}^{d}}-2\log \epsilon_v'\left(\binom{d+2}{2}\right)\\
&\geq -2dh(Q)-4\log (d+2)
\end{align*}
by the product formula.  So
\begin{align*}
\sum_{v\in T}h_{\alpha ,v}(\phi(P))\geq -2dh(Q)-8\log (d+2)-(d+2)\log 2-\sum_{v\in T}\log\max\{\|x-x_0\|_v,\|y-y_0\|_v\}.
\end{align*}


Then
\begin{equation*}
\sum_{v\in T}h_{Q,v}(P)\leq \sum_{v\in T} h_{\alpha ,v}(\phi(P))+(2d+2)h(Q)+8\log(d+2)+(d+3)\log 2.
\end{equation*}

Finally, suppose that
\begin{equation*}
\max\{|x-x_0|_v,|y-y_0|_v\}\geq C_v,
\end{equation*}
where
\begin{equation*}
C_v=\frac{1}{\epsilon_v((d+2)^42^{d+1})}\min\left\{ \frac{|f_1(x_0,y_0,1)|_v}{|f_1|_v\max\{|x_0|_v,|y_0|_v,1\}^{d}},\frac{|f_2(x_0,y_0,1)|_v}{|f_2|_v\max\{|x_0|_v,|y_0|_v,1\}^{d}}\right\}.
\end{equation*}
As noted before, $C_v\leq 1$.  Then one easily finds that
\begin{align*}
\frac{\max \{|x|_v,|y|_v,1\}}{\max\{|x-x_0|_v,|y-y_0|_v\}}&=\frac{\max \{|(x-x_0)+x_0|_v,|(y-y_0)+y_0|_v,1\}}{\max\{|x-x_0|_v,|y-y_0|_v\}}\\
&\leq \frac{\epsilon_v(2)\max\{|(x-x_0)|_v,|(y-y_0)|_v,|x_0|_v,|y_0|_v,1\}}{\max\{|x-x_0|_v,|y-y_0|_v\}}\\
&\leq \frac{\epsilon_v(2)\max\{|x_0|_v,|y_0|_v,1\}}{C_v}.
\end{align*}
So 
\begin{align*}
h_{Q,v}(P)&=\frac{[k_v:\mathbb{Q}_v]}{[k:\mathbb{Q}]}\log \frac{\max \{|x_0|_v,|y_0|_v,1\}\max \{|x|_v,|y|_v,1\}}{\max\{|x-x_0|_v,|y-y_0|_v,|x_0y-y_0x|_v\}}\\
&\leq \frac{[k_v:\mathbb{Q}_v]}{[k:\mathbb{Q}]}\log \epsilon_v(2)\max \{|x_0|_v,|y_0|_v,1\}^2/C_v\\
&\leq \frac{[k_v:\mathbb{Q}_v]}{[k:\mathbb{Q}]}(2\log \max \{|x_0|_v,|y_0|_v,1\}+\log \epsilon_v(2)-\log C_v).
\end{align*}
Then using Lemma \ref{lhb}, we find
\begin{equation*}
\sum_{v\in T} h_{Q,v}(P)\leq   \sum_{v\in T} h_{\alpha ,v}(\phi(P))+2h(Q)+2\log 2-\sum_{v\in T}\frac{[k_v:\mathbb{Q}_v]}{[k:\mathbb{Q}]}\log C_v.
\end{equation*}
Since
\begin{align*}
\sum_{v\in T}\frac{[k_v:\mathbb{Q}_v]}{[k:\mathbb{Q}]}\log C_v&\geq \sum_{v\in M_k}\frac{[k_v:\mathbb{Q}_v]}{[k:\mathbb{Q}]}\log C_v\\
&\geq \sum_{j=1}^2\sum_{v\in M_k}\log \frac{\|f_j(x_0,y_0,1)\|_v}{\epsilon_v'((d+2)^42^{d+1})\|f_j\|_v\max\{\|x_0\|_v,\|y_0\|_v,1\}^{d}}\\
&\geq -8\log(d+2)-(2d+2)\log 2-2dh(Q),
\end{align*}
where we have used the product formula in the last line,
we obtain
\begin{equation*}
\sum_{v\in T} h_{Q,v}(P)\leq \sum_{v\in T} h_{\alpha ,v}(\phi(P))+(2d+2)h(Q)+(2d+4)\log 2+8\log (d+2).
\end{equation*}
\end{proof}

\section{Explicit Results for $\mathbb{P}^2$}
\label{PP}

In this section we give a proof of Theorem \ref{thP2}.  The proof will follow the proof in Section \ref{SP}, except that we will give explicit estimates at each step.  We begin with an explicit version of Theorem \ref{B2}.

\begin{theorem}
\label{B3}
Let $k$ be a number field of degree $\delta$ and discriminant $\Delta$.  Let $S$ be a finite set of places of $k$, containing the archimedean places, of cardinality $s$ .  Let $C_1$ and $C_2$ be distinct curves over $k$ in $\mathbb{P}^2$ defined by homogeneous polynomials $f_1,f_2\in \O_k[x,y,z]$, respectively, of degrees $d_1$ and $d_2$, respectively.  Let $d=\max\{d_1,d_2\}$ and $\phi=\frac{f_1^{d_2}(x,y,z)}{f_2^{d_1}(x,y,z)}$, a rational function on $\mathbb{P}^2$.  Let $Q\in \mathbb{P}^2(\kbar)\setminus (C_1\cup C_2)$ and let $\delta'=[k(Q):\mathbb{Q}]$.  Let $w\in M_{k(Q)}$ and let $0<\epsilon<1$.  Then for all $P\in (\mathbb{P}^2\setminus (C_1\cup C_2))(\O_{k,S})$, either
\begin{equation}
\label{eh}
h_{Q,w}(P)\leq \epsilon h(P)+c_4(\epsilon,k,S,w,Q,C_1,C_2)
\end{equation}
or
\begin{equation*}
\phi(P)=\phi(Q),
\end{equation*}
where
\begin{align*}
c_4&=(2d^2+2)h(Q)+10\log(d^2+2)+(2d^2+7)\log 2+\frac{1}{d}\left(h(C_1)+h(C_2)\right)+\frac{\log |\Delta|}{\delta d^2}+\frac{2\delta}{d^2} c_3(k)+c_5,\\
c_5&=6.4(d^2c_2(\delta',s)/\epsilon)\frac{N(w)}{\log N(w)}c_6c_7\max \{\log((d^2c_2(\delta',s)/\epsilon)N(w)),\log^* c_6\},\\
c_6&=Q_S\left(1+s c_3(k,S)+\frac{1}{\delta}\log |\Delta|\right),\\
c_7&=d^2h(Q)+dh(C_1)+dh(C_2)+\frac{1}{\delta}\log |\Delta|+2\delta c_3(k)+2d^2\log 2+2\log (d^2+2).
\end{align*}
In particular, \eqref{eh} holds for all $P\in (\mathbb{P}^2\setminus (C_1\cup C_2))(\O_{k,S})$ outside of an effectively computable finite union of plane curves $Z$.
\end{theorem}

\begin{proof}
Let $I_1$ and $I_2$ be the ideals of $\O_k$ generated by the coefficients of $g_1=f_1^{d_2}$ and $g_2=f_2^{d_1}$, respectively.  We rescale $g_1$ and $g_2$ as in Lemma \ref{arch} \eqref{c1} and its proof (viewing the coefficients of the polynomials as giving points in projective space).   In particular, $N(I_1),N(I_2)\leq \sqrt{|\Delta|}$.   Let $\phi=\frac{g_1}{g_2}$ and let $P\in (\mathbb{P}^2\setminus (C_1\cup C_2))(\O_{k,S})$.  Then it follows from the definitions that we have an equality of fractional ideals $\phi(P)\O_{k}=\frac{I_1}{I_2}J$, where $J$ is a fractional ideal supported on the primes in $S$.  By Theorem \ref{abl}, we can write $\phi(P)=\beta u$, where $u\in \O_{k,S}^*$ and 
\begin{equation*}
h(\beta)\leq s c_3(k,S)+\frac{1}{\delta}\log N(I_1)+\frac{1}{\delta}\log N(I_2) \leq s c_3(k,S)+\frac{1}{\delta}\log |\Delta|.
\end{equation*}

Let $\alpha=\phi(Q)$ and suppose that $\phi(P)\neq \alpha$.  By Theorem \ref{Baker}, substituting $\frac{\epsilon}{d^2}$ for $\epsilon$ and taking $G$ to be the multiplicative group generated by $\beta$ and $\O_{k,S}^*$, we have the inequality
\begin{align*}
h_{\alpha,w}(\phi(P))&\leq \frac{\epsilon}{d^2} h(\phi(P))+c_1\left(\frac{\epsilon}{d^2}, k(Q),G,w,\alpha\right)+\log 2.
\end{align*}

Note that $\deg \phi\leq d^2$.  By Lemma \ref{lemn},

\begin{align*}
h_{Q,w}(P)&\leq h_{\alpha,w}(\phi(P))+(2d^2+2)h(Q)+8\log(d^2+2)+(2d^2+4)\log 2\\
&\leq \frac{\epsilon}{d^2} h(\phi(P))+(2d^2+2)h(Q)+8\log(d^2+2)+(2d^2+5)\log 2+c_1\left(\frac{\epsilon}{d^2}, k(Q),G,w,\alpha\right).
\end{align*}
By Lemma \ref{rh},
\begin{equation*}
h(\phi(P))\leq d^2h(P)+h(\phi)+\log\binom{d^2+2}{2}\leq d^2h(P)+h(\phi)+2\log (d^2+2).
\end{equation*}
Let $s_\infty$ be the number of archimedean places of $k$.  By Lemma \ref{arch} \eqref{c1} and the construction of $g_1$ and $g_2$,
\begin{align*}
h(\phi)&=\sum_{v\in M_k}\max\{\|g_1\|_v,\|g_2\|_v\}\leq \sum_{\substack{v\in M_k\\v|\infty}}\max\{\|g_1\|_v,\|g_2\|_v\}\\
&\leq s_\infty\left(\frac{1}{s_\infty}h(f_1^{d_2})+\frac{1}{2\delta s_\infty}\log |\Delta|+c_3(k)+\frac{1}{s_\infty}h(f_2^{d_1})+\frac{1}{2\delta s_\infty}\log |\Delta|+c_3(k)\right)\\
&\leq h(f_1^{d_2})+h(f_2^{d_1})+\frac{1}{\delta}\log |\Delta|+2\delta c_3(k).
\end{align*}
By Lemma \ref{ph},
\begin{equation*}
h(f_1^{d_2})+h(f_2^{d_1})\leq dh(f_1)+dh(f_2)+2d^2\log 2=dh(C_1)+dh(C_2)+2d^2\log 2.
\end{equation*}
So
\begin{equation*}
h(\phi)\leq dh(C_1)+dh(C_2)+\frac{1}{\delta}\log |\Delta|+2\delta c_3(k)+2d^2\log 2.
\end{equation*}

Then
\begin{multline*}
h_{Q,w}(P)<\epsilon h(P)+(2d^2+2)h(Q)+10\log(d^2+2)+(2d^2+7)\log 2\\
+\frac{1}{d}\left(h(C_1)+h(C_2)\right)+\frac{1}{\delta d^2}\log |\Delta|+\frac{2\delta}{d^2} c_3(k)+ c_1\left(\frac{\epsilon}{d^2}, k(Q),G,w,\alpha\right).
\end{multline*}
Finally, we can estimate the last term using
\begin{equation*}
Q_G\leq Q_S \max\{h(\beta),1\}\leq Q_S\left(1+s c_3(k,S)+\frac{1}{\delta}\log |\Delta|\right)
\end{equation*}
and, using Lemma \ref{rh} again,
\begin{align*}
h(\alpha)&\leq d^2h(Q)+h(\phi)+2\log (d^2+2)\\
&\leq d^2h(Q)+dh(C_1)+dh(C_2)+\frac{1}{\delta}\log |\Delta|+2\delta c_3(k)+2d^2\log 2+2\log (d^2+2).
\end{align*}

\end{proof}

We now prove Theorem \ref{thP2}.

\begin{proof}[Proof of Theorem \ref{thP2}]
Let $d_i=\deg C_i$, $i=1,\ldots, n$.  Let $P\in \left(\mathbb{P}^2\setminus \cup_{i=1}^nC_i\right)(\O_{k,S})$.  Then
\begin{equation*}
\sum_{v\in S}h_{C_i,v}(P)=d_ih(P), \quad i=1,\ldots, n.
\end{equation*}
So for each $i$, there exists a place $v\in S$ such that $h_{C_i,v}(P)\geq \frac{1}{s}h(P)$.  Since $s<n$, there exists a place $v\in S$ and distinct elements $i,j\in\{1,\ldots, n\}$ such that
\begin{equation*}
\min \{h_{C_i,v}(P),h_{C_j,v}(P)\}\geq \frac{1}{s}h(P).
\end{equation*}
The theorem is then a consequence of the following lemma.

\begin{lemma}
Let $k$ be a number field of degree $\delta$ and discriminant $\Delta$.  Let $S$ be a finite set of places of $k$, containing the archimedean places, of cardinality $s$.  Let $C_1, \ldots, C_n\subset \mathbb{P}^2$ be distinct curves over $k$ such that at most $n-2$ of the curves $C_i$ intersect at any point of $\mathbb{P}^2(\kbar)$.  Let $d_i=\deg C_i$, $d=\max_i d_i$, $h=\max_i h(C_i)$, and $N=\max_{v\in S}N(v)$.  Let $Z'$ be the set from Theorem \ref{thP2}.  Let $0<\epsilon<1$ and $v\in S$.  Then any point $P\in \left(\mathbb{P}^2\setminus \cup_{i=1}^nC_i\right)(\O_{k,S})$ with
\begin{equation*}
\min \{h_{C_1,v}(P),h_{C_2,v}(P)\}\geq \epsilon h(P)
\end{equation*}
satisfies either $P\in Z'$ or
\begin{equation*}
h(P)<2^{20s+4\delta+75}d^{6s+34}\delta^{5s+8\delta-3}s^{4s-1}N^{d^2}(\log^* N)^{2s}|\Delta|^{3/2}(\log^*|\Delta|)^{3\delta}(h+1)/\epsilon^3.
\end{equation*}
\end{lemma}

\begin{proof}
Let $(C_1\cap C_2)(\kbar)=\{Q_1,\ldots, Q_r\}\subset \mathbb{P}^2(\kbar)$ and let $Q_i=(x_i,y_i,z_i)$, $x_i,y_i,z_i\in \O_{k(Q_i)}$, $i=1,\ldots, r$, where $r\leq d^2$.  Let $L=k(Q_1,\ldots, Q_r)$.  We note that $[k(Q_i):k]\leq d^2$ for all $i$.  Let $C_i$ be defined by $f_i\in \O_k[x,y,z]$, $i=1,\ldots, n$, and let
\begin{equation*}
h_\infty=\log \max_{\substack{w\in M_L\\ w|\infty}} \left\{|f_1|_w, |f_2|_w, \max \left|\prod_{i=1}^rg_i\right|_w\right\},
\end{equation*}
where the max is taken over all possible choices of 
\begin{equation*}
g_i\in \{z_ix-x_iz,z_iy-y_iz,x_iy-y_ix\}\subset \O_L[x,y,z], \quad i=1,\ldots, r.  
\end{equation*}
Now fix a choice of $g_i\in \{z_ix-x_iz,z_iy-y_iz,x_iy-y_ix\}$, $i=1,\ldots, r$.  Since $\prod_{i=1}^rg_i$ vanishes at all the points $Q_i$, by the effective Hilbert Nullstellensatz (see Remark~\ref{RHN}), there exists a positive integer $M$, homogeneous polynomials $a_1, a_2\in \O_L[x,y,z]$ with $\deg a_1= rM-\deg f_1, \deg a_2=rM-\deg f_2$, and a constant $a\in \O_L$ such that
\begin{equation*}
f_1(x,y,z)a_1(x,y,z)+f_2(x,y,z)a_2(x,y,z)=a\left(\prod_{i=1}^r g_i\right)^M
\end{equation*}
and
\begin{align*}
M&\leq (8d)^8,\\
\log \max_{\substack{w\in M_L\\ w|\infty}} \{|a_1|_w,|a_2|_w, |a|_w\}&\leq (8d)^{15}(h_\infty+8d\log 8d).
\end{align*}
Let $w$ be a place of $L$ lying above $v$ (we will choose a specific such $w$ later).  Let $x,y,z\in k$.  It follows that there exists $a_1, a_2, a$, and $M$, as above, such that
\begin{align*}
&\left(\prod_{i=1}^r \max\{|z_ix-x_iz|_w,|z_iy-y_iz|_w,|x_iy-y_ix|_w\right)^M\\
&\qquad \qquad=\frac{1}{|a|_w}|f_1(x,y,z)a_1(x,y,z)+f_2(x,y,z)a_2(x,y,z)|_w\\
&\qquad \qquad\leq 2\max\{|f_1(x,y,z)a_1(x,y,z)|_w,|f_2(x,y,z)a_2(x,y,z)|_w\}/|a|_w\\
&\qquad \qquad\leq 2(rM)^2\max_{i=1,2}\left\{|f_i(x,y,z)|_w|a_i|_w\max\{|x|_w,|y|_w,|z|_w\}^{rM-\deg f_i}\right\}/|a|_w.
\end{align*}
So
\begin{multline*}
\left(\prod_{i=1}^r \frac{\max\{|z_ix-x_iz|_w,|z_iy-y_iz|_w,|x_iy-y_ix|_w}{\max\{|x|_w,|y|_w,|z|_w\}}\right)^M\leq\\ \frac{2(rM)^2}{|a|_w}\max\{|a_1|_w,|a_2|_w\}\max_{i=1,2} \frac{|f_i(x,y,z)|_w}{\max\{|x|_w,|y|_w,|z|_w\}^{\deg f_i}}.
\end{multline*}

Let $r_{w/v}=[L:k]/[L_w:k_v]$.  Taking logarithms, rearranging, and using the definitions and inequalities above, we find
\begin{multline}
\epsilon h(P)\leq \min\{h_{C_1,v}(P), h_{C_2,v}(P)\}=r_{w/v}\min\{h_{C_1,w}(P), h_{C_2,w}(P)\}\label{key}\\
\leq Mr_{w/v}\sum_{i=1}^rh_{Q_i,w}(P)-Mr_{w/v}\sum_{i=1}^r\log\max\{\|x_i\|_w,\|y_i\|_w, \|z_i\|_w\}+\log 2(rM)^2\\
\quad+r_{w/v}\log\max\{\|f_1\|_w,\|f_2\|_w\}+r_{w/v}\log\max\{\|a_1\|_w,\|a_2\|_w\}-r_{w/v}\log\|a\|_w.
\end{multline}

Let $Q_l\in (C_1\cap C_2)(\kbar)$.  Then by assumption, there exists $i,j\in\{1,\ldots,n\}$, $i\neq j$, such that $Q_l\not\in C_i\cup C_j$.  Let $w_l$ be the place of $k(Q_l)$ lying below $w$ and let $r_{w_l/v}=[k(Q_l):k]/[k(Q_l)_{w_l}:k_v]$.  Let $\Phi_P$ and $Z'$ be as in Theorem \ref{thP2}.

By Theorem \ref{B3}, either
\begin{equation*}
P\in \bigcap_{\phi\in \Phi_{Q_l}} \{Q\in X(\kbar): \phi(Q)=\phi(Q_l)\}\subset Z'
\end{equation*}
or
\begin{equation*}
r_{w/v}h_{Q_l,w}(P)=r_{w_l/v}h_{Q_l,w_l}(P)< \frac{\epsilon}{2rM} h(P)+\max_{i,j}c_4\left(\frac{\epsilon}{2rMr_{w_l/v}},k,S,w_l,Q_l,C_i,C_j\right)
\end{equation*}
for all $P\in \left(\mathbb{P}^2\setminus \cup_{i=1}^nC_i\right)(\O_{k,S})=\cap_{i,j} (\mathbb{P}^2\setminus (C_i\cup C_j))(\O_{k,S})$.

Suppose now that $P\not\in Z'$.  Summing over all points in $C_1\cap C_2$, we obtain
\begin{equation*}
Mr_{w/v}\sum_{l=1}^rh_{Q_l,w}(P)< \frac{\epsilon}{2} h(P)+\sum_{l=1}^r\max_{i,j}c_4\left(\frac{\epsilon}{2rMr_{w_l/v}},k,S,w_l,Q_l,C_i,C_j\right).
\end{equation*}
Substituting into \eqref{key} we find that
\begin{multline}
\label{kin}
h(P)<\frac{2}{\epsilon}\left(\sum_{l=1}^r\max_{i,j}c_4\left(\frac{\epsilon}{2rMr_{w_l/v}},k,S,w_l,Q_l,C_i,C_j\right)+\log 2(rM)^2+r_{w/v}\log\max_{i=1,2}\|f_i\|_w\right.\\
\left. +r_{w/v}\log\max_{i=1,2}\|a_i\|_w-Mr_{w/v}\sum_{i=1}^r\log\max\{\|x_i\|_w,\|y_i\|_w, \|z_i\|_w\}-r_{w/v}\log\|a\|_w\right).
\end{multline}
We now estimate all the terms on the right-hand side.  The dominant term, which comes from the first sum above, is 
\begin{equation}
\label{mt}
\sum_{l=1}^r\max_{i,j}c_5\left(\frac{\epsilon}{2rMr_{w_l/v}},k,S,w_l,Q_l,C_i,C_j\right).
\end{equation}
We estimate this term first.  We note that by \eqref{Rest}, \eqref{Qest}, and \eqref{c3est1},
\begin{equation*}
c_6(k,S)\leq 2^{3s+3}s^{4s-3}\delta^{2s+2\delta-5}(\log^* N)^{2s-\delta} |\Delta|(\log^*|\Delta|)^{2\delta-2}.
\end{equation*}
Then
\begin{align*}
2rMr_{w_l/v}d^2c_2(d^2\delta,s)N(w_l)c_6(k,S) &=72rMd^2r_{w_l/v}N(w_l)(16ed^2\delta)^{3s+5}(\log^* d^2\delta)^2c_6(k,S)\\
&\leq 2^{20s+61}d^{6s+28}\delta^{5s+2\delta+2}s^{4s-3}N^{d^2}(\log^* N)^{2s-\delta}|\Delta|(\log^*|\Delta|)^{2\delta-2}.
\end{align*}
Simple estimates then also give
\begin{equation*}
\log 2rMr_{w_l/v}d^2c_2(d^2\delta,s)N(w_l)c_6(k,S)/\epsilon\leq 2^7s^2d^2(\log^*N)(\log^*|\Delta|)/\epsilon.
\end{equation*}
We have $\sum_{l=1}^rh(Q_l)\leq d(h(C_1)+h(C_2)+4d)\leq 4d^2(h+1)$, by Theorem \ref{TB}, and
\begin{align*}
\sum_{l=1}^r\max_{i,j}c_7\left(k,Q_l,C_i,C_j\right)&\leq \sum_{l=1}^r(d^2h(Q_l)+2dh+2d^2\log 2+2\log (d^2+2)+\frac{1}{\delta}\log |\Delta|+2\delta c_3(k))\\
&\leq 2^{4\delta+2} \delta^{6\delta-5}\sqrt{|\Delta|}(\log^*|\Delta|)^\delta d^4(h+1).
\end{align*}

Then \eqref{mt} is bounded by
\begin{equation}
\label{kb}
2^{20s+4\delta+73}d^{6s+34}\delta^{5s+8\delta-3}s^{4s-1}N^{d^2}(\log^* N)^{2s}|\Delta|^{3/2}(\log^*|\Delta|)^{3\delta}(h+1)/\epsilon^2.
\end{equation}

In the remainder of the proof, we will show that the sum of the remaining elements in the parentheses on the right-hand side of \eqref{kin} can also be bounded by this quantity.  Thus, we find that
\begin{equation*}
h(P)\leq 2^{20s+4\delta+75}d^{6s+34}\delta^{5s+8\delta-3}s^{4s-1}N^{d^2}(\log^* N)^{2s}|\Delta|^{3/2}(\log^*|\Delta|)^{3\delta}(h+1)/\epsilon^3,
\end{equation*}
proving the lemma.

First, we handle the remaining terms coming from the first sum in \eqref{kin}:

\begin{align*}
\sum_{l=1}^r(2d^2+2)h(Q_l)+10r\log(d^2+2)+(2d^2+7)r\log 2+\frac{r}{d}\left(h(C_1)+h(C_2)\right)+\frac{r}{\delta d^2}\log |\Delta|+\frac{2\delta r}{d^2} c_3(k)\\
\leq 4(2d^2+2)d^2(h+1)+10d^2\log(d^2+2)+(2d^2+7)d^2+2dh+\frac{1}{\delta}\log |\Delta|+2\delta 2^{4\delta} \delta^{6\delta}\sqrt{|\Delta|}(\log^*|\Delta|)^\delta\\
\leq 55d^4(h+1)+2^{4\delta+2} \delta^{6\delta+1}\sqrt{|\Delta|}(\log^*|\Delta|)^\delta.
\end{align*}

We now bound $h_\infty$, after making some appropriate choices.  Choose $f_1, f_2$, and $Q_l=(x_l,y_l,z_l)$, $l=1,\ldots, r$, as in Lemma \ref{arch} \eqref{c2}.  Then
\begin{align*}
-\left((d^2\delta)^2 h(Q_l)+\frac{(d^2\delta)}{2}\log (d^2\delta)\right) &\leq \log\max\{\|x_l\|_{w_l},\|y_l\|_{w_l}, \|z_l\|_{w_l}\}\leq (2d^2\delta+1)h(Q_l)+\log d^2\delta,\\
\log \max_{\substack{v\in M_k\\ v|\infty}} |f_1|_v&\leq (2\delta+1)\delta h(C_1)+\delta \log \delta\leq 3\delta^2(h+1),\\
\log \max_{\substack{v\in M_k\\ v|\infty}} |f_2|_v&\leq (2\delta+1)\delta h(C_2)+\delta \log \delta\leq 3\delta^2(h+1).
\end{align*}

Let $g_l\in \{z_lx-x_lz,z_ly-y_lz,x_ly-y_lx\}, i=1,\ldots, r$.  If $w\in M_L$, $w|\infty$, then
\begin{align*}
\log \left|\prod_{l=1}^r g_l\right|_w&\leq \log 2^r \prod_{l=1}^r |g_l|_w\leq r\log 2+\log\sum_{l=1}^r |g_l|_w\leq r\log 2+\sum_{l=1}^r \log \max\{|x_l|_{w_l}, |y_l|_{w_l}, |z_l|_{w_l}\}\\
&\leq d^2\log 2+\sum_{l=1}^r ((2d^2\delta+1)d^2\delta h(Q_l)+d^2\delta\log d^2\delta)\\
&\leq d^2\log 2+4(2d^2\delta+1)d^4\delta(h+1)+d^4\delta\log d^2\delta.
\end{align*}
Since $w\in M_L$ was arbitrary, we have
\begin{align*}
\log \max_{\substack{w\in M_L\\ w|\infty}} \left|\prod_{l=1}^r g_l\right|_w\leq d^2\log 2+4(2d^2\delta+1)d^4\delta(h+1)+d^4\delta\log d^2\delta.
\end{align*}
It follows easily that
\begin{equation*}
h_\infty\leq 14d^6\delta^2(h+1).
\end{equation*}

Then from the above, we have
\begin{align*}
\log 2(rM)^2&\leq \log 2d^4(8d)^{16}\leq 2^6d,\\
r_{w/v}\log\max\{\|f_1\|_w,\|f_2\|_w\}&\leq \log\max\{|f_1|_w,|f_2|_w\}\leq 4\delta^2(h+1),\\
r_{w/v}\log\max\{\|a_1\|_w,\|a_2\|_w\}&\leq \log\max\{|a_1|_w,|a_2|_w\}\leq (8d)^{15}(h_\infty+8d\log 8d)\\
&\leq 2^{51}d^{21}\delta^2(h+1).
\end{align*}

We also find
\begin{align*}
-Mr_{w/v}\sum_{l=1}^r\log\max\{\|x_l\|_w,\|y_l\|_w,\|z_l\|_w\}&=-M\sum_{l=1}^rr_{w_l/v}\log\max\{\|x_l\|_{w_l},\|y_l\|_{w_l},\|z_l\|_{w_l}\}\\
&\leq (8d)^8d^2\left((d^2\delta)^2 \sum_{l=1}^r h(Q_l)+\frac{d^4\delta}{2}\log (d^2\delta)\right)\\
&\leq 2^{24}d^{10}\left(4d^6\delta^2(h+1)+\frac{d^4\delta}{2}\log (d^2\delta) \right)\\
&\leq 2^{27}d^{16}\delta^2(h+1).
\end{align*}

From the product formula and the fact that $a\in O_L$, we have the inequality
\begin{align*}
-\sum_{\substack{w'\in M_L\\w'|v}}\log \|a\|_{w'}&=\sum_{\substack{w'\in M_L\\w'\nmid v}}\log\|a\|_{w'}\leq \sum_{\substack{w'\in M_L\\w'| \infty}}\max\{\log\|a\|_{w'},0\}\\
&\leq(8d)^{15}(h_\infty+8d\log 8d)\leq 2^{51}d^{21}\delta^2(h+1).
\end{align*}

Since $L/k$ is Galois, there are exactly $r_{w/v}$ places $w'\in M_L$ with $w'|v$.  Therefore, there exists a place $w'\in M_L$ with $w'|v$ and $-\log \|a\|_{w'}\leq \frac{1}{r_{w/v}}2^{51}d^{36}\delta^2(h+1)$.  Choosing now $w=w'$, we have

\begin{equation*}
-r_{w/v}\log \|a\|_{w}\leq 2^{51}d^{21}\delta^2(h+1).
\end{equation*}

Summing all of the inequalities above, we find that, as claimed, the remaining terms in \eqref{kin} are easily bounded by \eqref{kb}.

\end{proof}
\end{proof}

\subsection*{Acknowledgments}

The author would like to thank Mike Bennett for sharing and discussing his preprint \cite{Ben}.

\bibliography{linforms}
\end{document}